\documentclass[reqno, 11pt]{amsart}
\usepackage{amsmath}
\usepackage{graphicx}
\usepackage{amsfonts}
\usepackage{amssymb}
\usepackage{color}
\usepackage{mathrsfs}
\usepackage{epsfig}
\usepackage{url}
\usepackage{mathrsfs,a4wide}
\usepackage{calligra}
\usepackage{amsmath}
\usepackage{amssymb}
\usepackage{graphicx}
\usepackage{amsthm}
\usepackage{thmtools}
\usepackage{enumitem}
\usepackage{hyperref, cleveref}
\usepackage{mathrsfs}
\usepackage{tikz}

\theoremstyle{plain}

\theoremstyle{definition}
\theoremstyle{remark}
\numberwithin{equation}{section}

\theoremstyle{plain} \declaretheorem[numberwithin = section, name = Theorem,
 refname = {Theorem}, Refname = {Theorem}]{thm}
\theoremstyle{plain} \declaretheorem[numberlike = thm, name = Proposition,
 refname = {Proposition}, Refname = {Proposition}]{prop}
\theoremstyle{plain} \declaretheorem[numberlike = thm, name = Lemma,
refname = {Lemma}, Refname = {Lemma}]{lem}
\theoremstyle{plain} 
\theoremstyle{definition} 
\theoremstyle{definition} \declaretheorem[numberlike = thm, name = Remark,
refname = {Remark}, Refname = {Remark}]{rem}
\theoremstyle{plain} \declaretheorem[numberlike = thm, name = Corollary,
refname = {Corollary}, Refname = {Corollary}]{cor}

\DeclareMathOperator {\R}{\mathbb{R}}
\DeclareMathOperator {\BV}{BV}
\DeclareMathOperator{\weak*}{\begin{picture}(10,4)
							\put(0,-2){$\rightharpoonup$}
							\put(3,3){$\ast$}
							\end{picture}}
\DeclareMathOperator {\h}{\mathcal{H}}

\DeclareMathOperator{\gammaliminf}{\Gamma\text{-}\liminf}

\DeclareMathOperator{\conv}{\text{conv}}

\title[Nonlocal phase transitions]
{$\Gamma$-convergence for nonlocal phase transitions involving the $H^{1/2}$ norm}

\author[T. Heilmann] {T. Heilmann}
\address[Tim Heilmann]{
	Zentrum Mathematik - M7, Technische Universit\"at M\"unchen, 
	Boltzmannstrasse 3, 85748 Garching, Germany	}
\email{heilmant@ma.tum.de}

\begin{document}
	
\begin{abstract}
We study functionals
\begin{equation*}
	F_\varepsilon (u) := \lambda_\varepsilon \int_\Omega W(u)  \, dx +
		\varepsilon \|u\|_{H^{1/2}}^2
\end{equation*}
for a double well potential $W$ and the Gagliardo seminorm $\|\cdot\|_{H^{1/2}}$ when
$\varepsilon \ln(\lambda_\varepsilon) \rightarrow k$ as $\varepsilon \rightarrow 0^+$
and show compactness in the space of $BV$ functions on $\Omega$ and the $\Gamma$-convergence
to the classical surface tension functional.
\end{abstract}

\maketitle

\section {Introduction}
In this article we consider a family of energy functionals
\begin{equation} \label{introduction_temp1}
	F_\varepsilon (u) := \lambda_\varepsilon \int_\Omega W(u)  \, dx +
		\varepsilon \int_{\Omega \times \Omega} \frac{(u(y) - u(x))^2}
		{|y - x|^{N+1}} \, d(y,x)
\end{equation}
defined on the space of $L^1$- functions on a domain $\Omega \subset \R^N$
and study their compactness and $\Gamma$-convergence properties as $\varepsilon \rightarrow 0^+$
and $\lambda_\varepsilon \sim e^{1/\varepsilon}$.
Here
$W$ is a double-well potential and we can see $F_\varepsilon$ as a nonlocal version of the classical
Modica-Mortola energy functional studied in \cite{m, mm}, the squared $L^2$- norm of the gradient being
replaced by the squared Gagliardo seminorm $\|u\|_{H^{1/2}}^2$.
Nonlocal Modica-Mortola type energies have been studied in \cite{ab,ab2}. There, the gradient term
is replaced by a convolution term and the energy functionals have the form
\begin{equation}\label{introduction_temp2}
	\frac{1}{\varepsilon} \int_\Omega W(u)  \, dx +
		\varepsilon \int_{\Omega \times \Omega} J_\varepsilon(y-x) \frac{(u(y) - u(x))^2}
		{\varepsilon^2} \, d(y,x) \,.
\end{equation}
If the convolution  kernels $J_\varepsilon(\cdot) = \varepsilon^{-n} J(\cdot / \varepsilon)$
satisfy $\int_{\R^N} J(h) (|h| \wedge |h|^2) \,dh < \infty$, in \cite{ab} it is shown that,
as $\varepsilon \rightarrow 0^+$, the
limit energy functional is finite on the set of those BV functions on $\Omega$ that
attain only the values of the wells of $W$
and is given via a cell formula depending on $J$ resulting in a possibly anisotropic perimeter functional.
Let us note that the convolution kernel $h \mapsto h^{-(N+1)}$ in (\ref{introduction_temp1})
does not satisfy the above boundedness assumption and it can be shown that
the cell formula would give the value $+ \infty$. Therefore, it is
reasonable to use another scaling in (\ref{introduction_temp1}), and we will consider the case
that $\lambda_\varepsilon \sim e^{1 / \varepsilon}$. With this choice we will obtain as $\Gamma$-limit
still an isotropic perimeter type functional. This result has already been shown in \cite{sv1}
with a proof that uses an energy lower bound from \cite{psv} to prove a compactness statement for
sequences of equibounded energy, while it exploits compactness 
and the existence of minimizers for an associated optimal profile problem from \cite{sv2}
to show $\Gamma$-convergence.
Our aim is to give a simplified self-contaied proof which may be easier to adapt to the case
of more general kernels. We would like to remark that a corresponding
compactness and $\Gamma$-convergence result for the space dimension $N = 1$ is shown in
\cite{abs}. The latter relies on a rearrangement inequality, a strategy that we were
not able use in higher dimension.

The plan of this article is as follows: In the following parts of the introduction we
summarize the notation that we use, state our main result
and give a short summary of the proof strategy. 
The proof of the main theorem is contained in section 2. More specifically, in subsection 2.1 we give
estimates for the nonlocal part of the energy functional and
in subsection 2.2 lower and upper bounds for the energy functional in special situations.
Subsection 2.3 contains the proof of our compactness result and in
subsection 2.4 we show the $\Gamma$-convergence result.

For the theory of BV functions resp. $\Gamma$-convergence used in this article
we refer to the books \cite{afp} resp. \cite{dm}.

\subsection{Notation}
We consider sets in $\R^N$ for $N \geq 2$ equipped with its standard basis $e_1, \dots, e_N$.
We write $\conv(S)$ for the convex hull of a set $S \subset \R^N$ and
$S^C := \R^N\setminus  S$ for its complement. We write $S \Delta T$ for the symmetric difference of two
 sets $S$ and $T$; sets in $\R^{N-1}$ are identified with the corresponding subsets in the hyperplane 
$\{x_N = 0\} \subset \R^N$.
For a point $x = (x_1, \dots, x_N) \in \R^N$ we write $x' := (x_1,\dots,x_{N-1})$ and also $x = (x', x_N)$.
Given a measurable set $S \subset \R^N$ and $x' \in \R^{N-1}$ 
we write $S_{x'} := \{x'\} \times \R \, \cap \, S$
for its section in direction $e_N$ at $x'$.
Moreover, we write $-S := \{x \in \R^N \, : \, -x \in S \}$ and,
given in addition $x \in \R^N$, $S + x := \{ y + x \,:\, y \in S \}$.
We write $Q_R(x) := x + [-R/2, R/2]^N$ for the $N$-cube with side length $R$ centered at $x \in \R^N$
and $Q_R'(x') := x' + [-R/2, R/2]^{N-1}$ for the corresponding $N-1$-cube.
Similarly we write $B_R(x) := \{y \in \R^N \,:\, |x-y| < R\}$ and
$B_R'(x') := \{y' \in \R^{N-1} \,:\, |x'-y'| < R\}$.
We write $C$ for constants that may change their value from line to line.
In order to shorten the notation, we write integrals of functions $f$ defined on the product of two sets
$A,B$ as
$\int_{A \times B} f(x,y) \,d(x,y) := \int_A \int_B f(x,y) \,dy \,dx$ and write for the level sets of $f$
$\{ f(x,y) = a \} := \left\{ (x,y) \in A \times B \,:\, f(x,y) = a  \right\}$.
For a function of bounded variation $u \in \BV (\Omega)$ we write $S_u$ for its approximate jump set,
as defined e.g. in \cite{afp}.

\subsection{The energy functional, the main theorem and proof ideas}
We fix a bounded regular domain $\Omega \subset \R^N$ and consider the energy
$F_\varepsilon : L^1(\Omega) \rightarrow [0,+ \infty]$ defined as
\begin{align*}
	&F_\varepsilon (u_\varepsilon) := \lambda_\varepsilon \int_\Omega W(u_\varepsilon)  \, dx +
		\varepsilon \int_{\Omega \times \Omega} \frac{(u_\varepsilon(y) - u_\varepsilon(x))^2}
		{|y - x|^{N+1}} \, d(y,x) \,,
\end{align*}
where $\varepsilon > 0$ and $\lambda_\varepsilon$ is a parameter such that it exists
\begin{equation*}
	k: = \lim_{\varepsilon \rightarrow 0^+}\varepsilon \log(\lambda_\varepsilon) \in (0,\infty)\,.
\end{equation*}
We assume that the potential $W : \R \rightarrow [0, +\infty)$ satisfies for some $\alpha, \beta \in \R$,
$\alpha < \beta$
\begin{align*}
	&W(x) = 0 \, \Leftrightarrow \, x \in \{\alpha, \beta\} \\
	&W \text{ has at least linear growth at } \pm \infty\,,
\end{align*}
define the dimensional constant
\begin{equation*}
	\omega_{n-1} :=\h^{N-1} \left( B_1'(0)\right)
\end{equation*}
and the limit energy functional
\begin{equation*}
	F: L^1(\Omega) \rightarrow [0,\infty] \,,\, u \mapsto \left\{
		\begin{array}{lr}
			2(\beta - \alpha)^2\omega_{n-1} k \h^{N-1}(S_u) \,, &u \in \BV(\Omega, \{ \alpha,\beta \}) \\
			+\infty, &\text{otherwise.}
		\end{array} \right.
\end{equation*}
Recall that for a function $u \in \BV(\Omega, \{ \alpha,\beta \})$, $S_u$ can be seen
equivalently as the reduced boundary of the set $\{ u = \beta \}$.

Our main theorem is the following compactness and $\Gamma$-convergence result.
The proof of its first part is given in Proposition \ref{compactness}, while its second part is an immediate
consequence of Proposition \ref{liminf} and Proposition \ref{recovery}.
\begin{thm}
If $F_\varepsilon, F$ and $W$ are as defined above, the following statements hold:
	\newline
		Compactness: Given sequences $(u_{\varepsilon_h})_h$ in $L^1(\Omega)$
		and $(\varepsilon_h)_h$ in $\R$ such that $\varepsilon_h \rightarrow 0^+$ 
		and $F_{\varepsilon_h}(u_{\varepsilon_h})$ is uniformly bounded,
		then $(u_{\varepsilon_h})_h$ is relatively compact in $L^1(\Omega)$ and each of its cluster points
		belongs to $ \BV \left( \Omega, \{\alpha,\beta\} \right)$.
		\newline
	$\Gamma$-convergence: The functionals $F_\varepsilon$ $\Gamma$-converge in
		$L^1(\Omega)$ to the functional $F$ as $\varepsilon \rightarrow 0^+$.
\end{thm}

\begin{rem}
Note that the theorem holds true also for $N=1$ and it could be shown with a proof as the one given below,
the only difference being the calculation of $G$ in Lemma \ref{planeplane}
and Lemma \ref{correction_cylinderbound1}.
\end{rem}

In the following, we consider the energy
\begin {equation*}
	G(A,B) = \int_{A \times B} \frac{1}{|y-x|^{N+1}}\, d(y,x) \,,
\end{equation*}
where $A,B \subset \R^N$ are two disjoint sets. Heuristically, in the energy $F_\varepsilon (u_\varepsilon)$,
the first integral forces $u_\varepsilon$ to take values in $\{ \alpha,\beta \}$
as $\varepsilon \rightarrow 0^+$, and the second integral will be approximated by 
$G\left( \{ u_\varepsilon \approx \alpha \}, \{ u_\varepsilon \approx \beta \} \right)$.
Note that  it will even be only this second integral which contributes to the energy of a recovery sequence
(see the proof of Lemma \ref{recoverycube}).

Let us sketch the proof of the main theorem:
In subsection \ref{temp10} we estimate $G$ from below for two coaxial cylinders  $S' \times(d/2,l/2)$ and
$S' \times (-l/2,-d/2)$ exploiting the symmetries of the kernel.

This lower bound will give a lower bound as $\varepsilon$ tends to zero for the energy
$F_\varepsilon(u_\varepsilon)$ by a constant factor multiplied with
the Hausdorff measure of $S'$, whenever
$A_\varepsilon := \{ x \,:\, u_\varepsilon(x) < \alpha + \delta \}$ is of the form
$S' \times (1/\lambda_\varepsilon, l/2)$ and
$B_\varepsilon := \{ x \,:\, u_\varepsilon(x) > \beta - \delta \}$ is of the form
$S' \times (-l/2,-1/\lambda_\varepsilon)$ (for some fixed $\delta > 0$).

Once it is shown that for a sequence $(u_\varepsilon)_\varepsilon$ we can replace inside little cubes
$A_\varepsilon$ and $B_\varepsilon$ with sets of the type $S' \times \pm(1/\lambda_\varepsilon, l/2)$
without increasing the energy, we will obtain the  $\liminf$-inequality, Proposition \ref{liminf},
by a standard blow-up argument. Similarly, the proof of the compactness statement,
Proposition \ref{compactness}, follows
from a lower bound for $F_\varepsilon(u_\varepsilon)$ on little cubes on which
the mass of both $A_\varepsilon$ and $B_\varepsilon$ is bounded away from zero. 
These energy estimate results will follow from Corollary \ref{correction_cylinderbound2}.
We refer to the beginning of subsection \ref{temp7} for its proof idea.

The $\limsup$-inequality will follow from a localization and standard approximation argument:
in Proposition \ref{recovery}, we will reduce the general case to
finding  a recovery sequence $(u_\varepsilon)$ on a cylinder which approximates $\alpha$ 
on the upper half and $\beta$ on the lower half of this cylinder, and calculating its energy.
The latter is done in Lemma \ref{recoverycube}.

\section{Proof of the main result}

\subsection{Estimates for the nonlocal energy}\label{temp10}
In this subsection, we consider an energy defined on two disjoint sets related to
the nonlocal part of our energy functional $F_\varepsilon$:
we define for measurable sets $A,B \subset \R^N$ disjoint and $S \subset \R^N$
\begin{equation*}
G(A,B,S) := \int_{A \cap S \, \times \, B \cap S} \frac{1}{|y-x|^{N+1}} \, d(y,x)
\end{equation*}
and
\begin{equation*}
G(A,B) := G(A,B,\R^N) \,.
\end{equation*}

In order to simplify computations, it is helpful to rewrite $G$ as an integration with respect to antipodal
points on families of spheres as follows:
\begin{rem}[Computation of G] \label{temp1}
We define $h: \R^{2N} \rightarrow \R^N$ as $h(x,y) := \frac{x+y}{2}$
and obtain by an application of the Coarea formula that
\begin{align*}
	G(A,B,S) &= 2^{N/2} \int_{\R^N}  \int_{ \{h(x,y) = a \} } \frac{\chi_{A \cap S} (x) \chi_{B \cap S} (y)}
		{|y-x|^{N+1}} \,d \h^N (y,x) \,da \\
	&=  2^{N/2} \int_{\conv \left(( A\cap S) \cup (B \cap S) \right)}  \int_{ \{h(x,y) = a \} }
		 \frac{ \chi_{A \cap S} (x) 
		\chi_{B \cap S} (y)}{|y-x|^{N+1}} \,d \h^N (y,x) \,da.
\end{align*}
In order to shorten the notation, we also write the inner integral in the case $S = \R^N$ as
\begin{equation*}
	H(A,B,a) := 2^{N/2} \int_{ \{ h(x,y) = a \} } \frac{\chi_{A} (x) \chi_{B} (y)}
		{|y-x|^{N+1}} \,d \h^N (y,x).
\end{equation*}
\end{rem}

\begin{lem}[Value of $H$ on a line for two fattened hyperplanes] \label{planeplane}
Given $a' \in \R^{n-1}$, $d,l > 0$ such that $d<l$ and the sets
\begin{equation*}
	A_\infty := (d/2, l/2) \times \R^{N-1} \quad \text{and} \quad B_\infty := (-l/2,-d/2) \times \R^{N-1} \,,
\end{equation*}
it holds that
\begin {equation*}
	\int_{\R} H(A_\infty, B_\infty, (a',a_N)) \, da_N = \omega_{n-1} \left( 2\ln((l+d)/4) - \ln(l/2) - \ln(d/2) \right) \,,
\end{equation*}
where $\omega_{n-1} = \h^{N-1}\left(B_1'(0) \right)$.
\end{lem}
\begin{proof}
Noting that that $H(A_\infty, B_\infty, (a',a_N)) = 0$ if $a_N > \frac{l/2 - d/2}{2}$,
we first compute $H(A_\infty,B_\infty, (a',a_N))$ for $a_N \in \left( 0, \frac{l/2 -d/2}{2} \right)$.
Setting
\begin{equation*}
	i_a: \R^N \times \R^N \rightarrow \R^2, (x,y) \mapsto (|x' -a'|, x_N - a_N) \,,
\end{equation*}
 for $(x,y) \in \R^N \times \R^N$ and $(u,v) \in \R^2$  it holds that $h(x,y) = (a',a_N)$ and $i_a(x,y) = (u,v)$
if and only if $x \in \partial B_u'(a') \times \{\ v + a_N \}$ and $y = (2a'-x', a_N-v)$.
It follows
\footnote{In the second equality we applied the Coarea formula
(note that if $h(x,y) = a$, then $i_a(x,y) = (|x'/2 -y'/2|, x_N/2 -y_N/2)$ such that the tangential Jacobian of
$i_a$ inside the level set $\{ h(x,y) = a \}$ is $1/2^{N/2}$), the third equality holds because both
$\chi_{A_\infty}$ and $\chi_{B_\infty}$ depend only on $x_N$ and the fourth inequality is
by a change to polar coordinates.
In the last equality we have used that $\omega_{n-1} = \h^{N-2}(\partial B_1'(0)) / (N-1)$.}
\begin{align*}
	H(A_\infty, B_\infty,(a',a_N)) &= 2^{N/2} \int_{ \{ h(x,y) = a \} } 
		\frac{\chi_{A_\infty} (x) \chi_{B_\infty} (y)}
		{|y-x|^{N+1}} \,d \h^N (y,x) \\
	&= 2^N \int_{\R^2} \int_{ \{ h(x,y) = a \} \cap \{i_a(x,y) = (u,v) \}} \frac{\chi_{A_\infty} (x)
		 \chi_{B_\infty} (y)}
		{|y-x|^{N+1}} \,d \h^{N-2} (y,x) \,d(u,v) \\
	&=2^N \int_{\R^2} \chi_{(d/2,l/2)} (a_N+v) \chi_{(-l/2,-d/2)}(a_N-v) \\
	&\quad\quad
		  \frac{1} {\left( 2|(u,v)|\right)^{N+1}} \h^{N-2}(\partial B_u'(a') \times \{ v + a_N \}) \,d(u,v) \\
	&= \int_{(0, \pi/2) \times(0,\infty)} \chi_{(d/2,l/2)}(r \sin \theta +a_N) \chi_{(-l/2,-d/2)}(a_N- r\sin \theta) \\
	&\quad \frac{1}{2 r^{N+1}} \h^{N-2} (\partial B_1'(0)) (r \cos \theta)^{N-2} r \, d(r,\theta) \\
	&= \frac{1}{2} \h^{N-2}(\partial B_1'(0)) \int_{(0,\pi/2) \times 
		\left( \frac{d/2 + a_N}{\sin \theta}, \frac{l/2 - a_N}{\sin \theta}\right)}
		\cos(\theta)^{N-2} \frac{1}{r^2} \,d(\theta,r) \\
	&= \frac{\omega_{n-1}}{2} \left( \frac{1}{d/2 + a_N} - \frac{1}{l/2 - a_N} \right) \,.
\end{align*}

The symmetry of $H$ implies that $H(A_\infty, B_\infty, (a',-a_N)) = H(A_\infty, B_\infty, (a',a_N))$ and
therefore we conclude:
\begin{align*}
	\int_{\R} H(A_\infty, B_\infty, (a',a_N)) \, da_N &= \omega_{n-1} \int_{ \left( 0, \frac{l-d}{4} \right) }
		 \frac{1}{d/2 + a_N} - \frac{1}{l/2 - a_N} \, da_N \\
	&= \omega_{n-1} \left( 2\ln((l+d)/4) - \ln(l/2) - \ln(d/2) \right) \,.
\end{align*}
\end{proof}

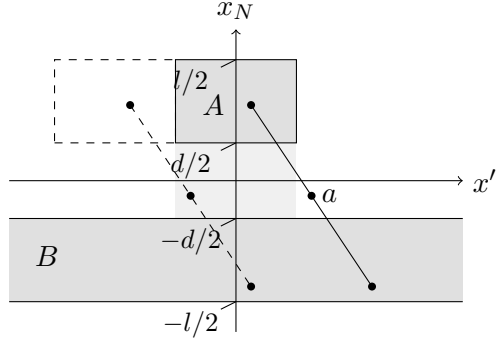
\begin{figure}
\begin{tikzpicture}
	\fill[fill = black!6] (-0.8,0.5) -- (0.8,0.5) -- (0.8,-0.5) -- (-0.8,-0.5) -- cycle;
	\filldraw[fill = black!12] (-0.8,0.5) -- (0.8,0.5) -- (0.8,1.6) -- (-0.8,1.6) -- cycle;
	\draw (-0.3, 1) node[anchor = center ]{$A$};
	\draw[dashed](-2.4,0.5) -- (-0.8,0.5) -- (-0.8,1.6) -- (-2.4,1.6) -- cycle;
	\fill[fill = black!12] (-3,-0.5) -- (3,-0.5) -- (3,-1.6) -- (-3,-1.6);
	\draw (-2.5, -1) node[anchor = center ]{$B$};
	\draw (-3,-0.5) -- (3,-0.5);
	\draw (-3,-1.6) -- (3,-1.6);
	\fill(0.2,1) circle [radius = 1.5pt];
	\fill(1,-0.2) circle [radius = 1.5pt];
	\draw (1,-0.2) node[anchor= west]{$a$};
	\fill(1.8,-1.4) circle [radius = 1.5pt];
	\draw(0.2,1) -- (1.8,-1.4);
	\fill(-1.4,1) circle [radius = 1.5pt];
	\fill(-0.6,-0.2) circle [radius = 1.5pt];
	\fill(0.2,-1.4) circle [radius = 1.5pt];
	\draw[dashed](-1.4,1) -- (0.2,-1.4);
	\draw[->](-3,0) --(3,0) node[anchor=west]{$x'$};
	\draw[->](0,-2) --(0,2) node[anchor=south]{$x_N$};
		\draw (0,1.6) --(-0.2,1.5);
		\draw (-0.6,1.3) node[anchor=center]{\small $l/2$};
		\draw (0,0.5) --(-0.2,0.4);
		\draw (-0.6,0.2) node[anchor=center]{\small $d/2$};
		\draw (0,-0.5) --(-0.2,-0.6);
		\draw (-0.6,-0.8) node[anchor=center]{\small $-d/2$};
		\draw (0,-1.6) --(-0.2,-1.7);
		\draw (-0.6,-1.9) node[anchor=center]{\small $-l/2$};
\end{tikzpicture}
\caption{Shifts of the cylinder $A$ allow to reduce to the situation of Lemma \ref{planeplane}}
\label{translation}
\end{figure}

We can now compute the value of $G$ for a cylinder and a fattened hyperplane by shifting the
cylinder parallel to the plane. A sketch of this situation is given in Figure \ref{translation}.
\begin{lem}[Value of $G$ for a cylinder and a fattened hyperplane] \label{cylinderplane}
Given $d,l,R > 0$ such that $d<l$ and the sets
\begin{equation*}
	A_R := Q_R'(0) \times (d/2,l/2) \quad \text{and} \quad B_\infty = \R^{N-1} \times (-l/2,-d/2) \,,
\end{equation*}
it holds that
\begin{equation*}
	 G(A_R,B_\infty) = R^{N-1} \omega_{n-1} \left( 2\ln((l+d)/4) - \ln(l/2) - \ln(d/2) \right) \,,
\end{equation*}
where $\omega_{n-1}$ is the constant defined in Lemma \ref{planeplane}.
\end{lem}
\begin{proof}
We compute
\begin{align*}
	G(A_R,B_\infty) &= \int_{\R^N} H(A_R,B_\infty,a) \,da = \sum_{h' \in \mathbb{Z}^{N-1}}
		\int_{Q_R'((Rh',0))} \int_{\R} H(A_R,B_\infty,(a',a_N)) \,da' \,d a_N \\
	&=  \sum_{h' \in \mathbb{Z}^{N-1}} \int_{Q_R'(0)} \int_{\R}
		H(A_R+(Rh',0),B_\infty,(a',a_N)) \,da' \,d a_N \\
	&= \int_{Q_R'(0)} \int_{\R} H(A_\infty,B_\infty,(a',a_N)) \,da_N \,da' \\
	&= R^{N-1} \omega_{n-1} \left( 2\ln((l+d)/4) - \ln(l/2) - \ln(d/2) \right),
\end{align*}
where in the last equality we used Lemma \ref{planeplane}.
\end{proof}

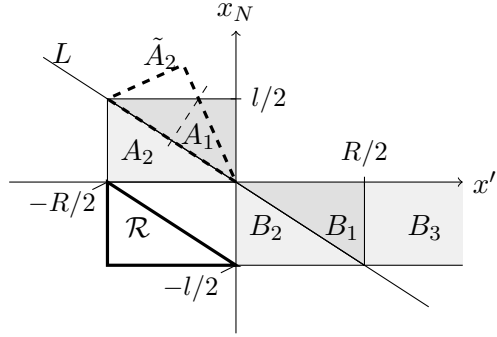
\begin{figure}
\begin{tikzpicture}
	\fill[fill = black!6](1.7,0) -- (1.7,-1.1) -- (3,-1.1) -- (3,0);
		\draw (1.7,-1.1)  -- (3,-1.1);
		\draw (2.5,-0.6) node[anchor = center]{$B_3$};
	\filldraw[fill = black!6](0,0) -- (1.7,-1.1) -- (0,-1.1);
		\draw (0.4,-0.6) node[anchor = center]{$B_2$};
	\filldraw[fill = black!12](0,0) -- (1.7,-1.1) -- (1.7,0);
		\draw (1.4,-0.6) node[anchor = center]{$B_1$};
	\filldraw[fill = black!6](-1.7,0) -- (0,0) -- (-1.7,1.1) -- cycle;
		\draw (-1.3,0.4) node[anchor = center]{$A_2$};
	\filldraw[fill = black!12] (0,0) -- (-1.7,1.1) -- (0,1.1);
		\draw (-0.5,0.6) node[anchor = center]{$A_1$};
	\draw[dashed, line width = 1.2pt] (0,0) -- (-1.7,1.1) -- (-0.696,1.551) -- cycle;
		\draw (-1,1.7) node[anchor = center]{$\tilde A_2$};
	\draw [dashed](-0.384,1.27) -- (-0.934,0.42);
	\draw (-2.55,1.65) -- (2.55,-1.65);
		\draw (-2.3,1.7) node[anchor = center]{$L$};
	\draw [line width = 1.2pt](-1.7,-1.1) -- (0,-1.1) -- (-1.7,0) -- cycle;
	\draw (-1.3,-0.6) node[anchor = center]{$\mathcal R$};
	\draw[->](-3,0) --(3,0) node[anchor=west]{$x'$};
	\draw[->](0,-2) --(0,2) node[anchor=south]{$x_N$};
	\draw (1.7,-2pt) --(1.7,2pt) node[anchor = south]{\small $R/2$};
	\draw (-1.9,-0.1) --(-1.7,0);
	\draw (-2.3,-0.3) node[anchor = center]{\small $-R/2$};
	\draw (-2pt,1.1) -- (2pt,1.1) node[anchor=west]{\small $l/2$};
	\draw (0,-1.1) --(-0.2,-1.2);
	\draw (-0.6,-1.4) node[anchor=center]{\small $-l/2$};
\end{tikzpicture}
\caption{The decomposition of $A$ and $B$ from the proof of Lemma \ref{cylindercomplement}
and the set $\mathcal R$ from Remark \ref{cylindercone}}
\label{subdivision}
\end{figure}

\begin{figure}
\begin{tikzpicture}
	\filldraw[fill = black!12] (-1.7,1.1) -- (1.7,1.1) -- (0,0) -- cycle;
	\draw (-0.3, 0.6) node[anchor = center ]{$D$};
	\fill[fill = black!12] (-3,0) -- (3,0) -- (3,-1.1) -- (-3,-1.1);
	\draw (-2.5, -0.6) node[anchor = center ]{$\hat B$};
	\draw (-3,0) -- (3,0);
	\draw (-3,-1.1) -- (3,-1.1);
	\draw[->](-3,0) --(3,0) node[anchor=west]{$x'$};
	\draw[->](0,-2) --(0,2) node[anchor=south]{$x_N$};
	\draw (1.7,-2pt) --(1.7,2pt) node[anchor = south]{\small $R/2$};
	\draw (-1.7,-2pt) --(-1.7,2pt) node[anchor = south]{\small $-R/2$};
	\draw (0,-1.1) --(-0.2,-1.2);
	\draw (-0.6,-1.4) node[anchor=center]{\small $-l/2$};
	\draw (0,1.1) --(-0.2,1.2);
	\draw (-0.6,1.4) node[anchor=center]{\small $l/2$};
\end{tikzpicture}
\caption{The sets $D$ and $\hat B$ from Lemma \ref{cylindercomplement}}
\label{triangleplane}
\end{figure}
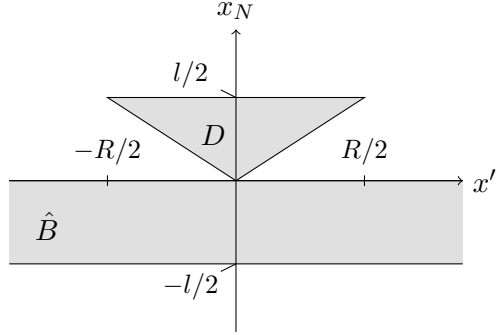

\begin{lem}[Estimate of $G$ for a cylinder and its complement inside a fattened hyperplane]
	 \label{cylindercomplement}
Given $d,l \in \R$ such that $0 \leq d<l \leq 8/3$ and the set $A_R = Q_R'(0) \times (d/2,l/2)$, it holds that
\begin{equation*}
	G \left(A_R, Q_R'(0)^C \times (-l/2,-d/2) \right) \leq C(N) R^{N-1} \left( 1 - \ln(l/2) \right) \,.
\end{equation*}
\end{lem}
\begin{proof}
Let us first note that the map $T: \R^{N-1} \times \R \rightarrow \R^{N-1} \times \R$, defined
as the one-homogeneous extension of the map $B_{1/2}'(0) \times \R \rightarrow Q_1'(0) \times \R$
mapping each point $(x',x_N)$ to the intersection of $\partial Q_1'(0) \times \{ x_N \}$ with
 the half-line with end-point $(0,x_N)$ and passing through $(x',x_N)$,
is Lipschitz with Lipschitz inverse. Therefore, we can estimate equivalently
\newline
$G \left( B_{R/2}'(0) \times (d/2,l/2), B_{R/2}'(0)^C \times (-l/2,-d/2) \right)$.
Similarly to the proof of Lemma \ref{planeplane}
we can now reduce to the two-dimensional case and obtain
\begin{equation} \label{temp2}
	G\left(A_R, Q_R'(0)^C \times \left(-\frac{l}{2},-\frac{d}{2}\right)\right) \leq
		C(N) R^{N-2} G \left( \left( -\frac{R}{2},0 \right) \times \left( \frac{d}{2}, \frac{l}{2} \right),
		(0,\infty) \times \left(-\frac{l}{2},-\frac{d}{2} \right) \right) \,:
\end{equation}
In fact, we set
$i: \R^N \rightarrow \R^2, \,i(x) := (|x'|,x_N)$ and compute
\footnote{In the second equality we used the Coarea formula for $x$. To obtain the first inequality
we wrote $y$ in cylindrical coordinates
with respect to the axis $\{x'\} \times \R$ and used the fact that, given $x \in B_{R/2}'(0) \times (d/2,l/2)$,
all points $y \in B_{R/2}'(0)^C \times(-l/2,-d/2)$ satisfy $|y' - x'| > R/2-|x'|$.}
\begin{align*}
	&G \left(B_{R/2}'(0) \times (d/2,l/2), B_{R/2}'(0)^C \times 
		(-l/2,-d/2)\right) \\
	&\quad = \int_{B_{R/2}'(0) \times (d/2,l/2)}
		\int_{B_{R/2}'(0)^C \times (-l/2,-d/2)} \frac{1}{|y-x|^{N+1}} \,dy\,dx \\
	&\quad = \int_{(0,R/2) \times (d/2,l/2)} \int_{ \{i(x) = (u,v)\} }
		\int_{B_{R/2}'(0)^C \times (-l/2,-d/2)} \frac{1}{|y-x|^{N+1}} \,dy\,d\h^{N-2}(x) \,d(u,v) \\
	&\quad \leq \int_{(0,R/2) \times (d/2,l/2)} \int_{ \{i(x) = (u,v)\} } C(N) \\
	&\quad \quad \quad \int_{(R/2- |x'|, \infty) \times (-l/2,-d/2)}
		\frac{ r^{N-2}}{|(r,y_N-x_N)|^{N+1}} \,d(r,y_N) \,d\h^{N-2}(x) \,d(u,v) \\
	&\quad \leq C(N) \int_{(0,R/2) \times (d/2,l/2)} \h^{N-2}( \partial B_u'(0) \times\{v\})
		\int_{(R/2- u, \infty) \times (-l/2,-d/2)} \frac{1}{|(r,y_N-v)|^3} \,d(r,y_N) \,d(u,v) \\
	&\quad \leq C(N) R^{N-2} G \left( (-R/2,0) \times (d/2,l/2), (0,\infty) \times (-l/2,-d/2) \right) \,.
\end{align*}

Let us note that, by the very definition of $G$,
\begin{equation*}
G \left( \left( -\frac{R}{2},0 \right) \times \left( \frac{d}{2}, \frac{l}{2} \right),
		(0,\infty) \times \left(-\frac{l}{2},-\frac{d}{2} \right) \right) \leq
G \left( \left( -\frac{R}{2},0 \right) \times \left( 0, \frac{l}{2} \right),
		(0,\infty) \times \left(-\frac{l}{2},0 \right) \right) \,.
\end{equation*}
It remains to estimate the right hand side of the previous
inequality. To this end we define
\begin{equation*}
A:=(-R/2,0) \times (0,l/2) \quad \text{and} \quad  B:= (0,\infty) \times (-l/2,0) 
\end{equation*}
and subdivide $A$ in two triangles $A_1,A_2$ and $B$ in the two triangles $B_1,B_2$
and the remainder $B_3$, as in Figure \ref{subdivision}. 
In addition, let us define $\tilde A_2$ to be the reflection of $A_2$ along the
line $L$ through $0$ and $(-R/2,l/2)$.
Given a point $y \in B_1 \cup B_3$ and a point $x \in A_2$, its reflection $\tilde x \in \tilde A_2$
satisfies $|y - \tilde x| \leq |y-x|$.
Note also that the reflection of $\tilde A_2$ along the dashed axis (of symmetry orthogonal to $L$)
equals $A_1$ and that for $\tilde x \in \tilde A_2 \setminus A_1$ its reflection
$\tilde {\tilde x} \in A_1 \setminus \tilde A_2$ satisfies $|y - \tilde{\tilde x}| \leq |y - \tilde x|$.
Therefore we obtain
\begin{align*}
	G(A_2,B_1 \cup B_3) &\leq G(\tilde A_2, B_1 \cup B_3) = G(\tilde A_2 \cap A_1, B_1 \cup B_3) + 
		G(\tilde A_2 \setminus A_1 , B_1 \cup B_3) \\
	&\leq  G(\tilde A_2 \cap A_1, B_1 \cup B_3) +G(A_1 \setminus \tilde A_2, B_1 \cup B_3) \\
	&= G(A_1, B_1 \cup B_3)\,.
\end{align*}
Moreover the pair of sets $(A_2,B_2)$ is obtained by rotating the pair $(B_1,A_1)$ by $\pi$ around the
origin which implies
\begin{equation*}
	G(A_2,B_2) = G(A_1,B_1) \leq G(A_1,B_1 \cup B_3) \,.
\end{equation*}
Summing up,
\begin{align} \label {temp3}
	G(A,B) &= G(A_1,B_1 \cup B_3) + G(A_1,B_2) +  G(A_2,B_1 \cup B_3) + G(A_2,B_2) \\
	&\leq 3  G(A_1,B_1 \cup B_3) + G(A_1,B_2) \leq 3 G(A_1,B) \nonumber \\
	&\leq \frac{3}{4} G(D, \hat B) \nonumber \,,
\end{align}
where we denoted $D$ for the triangle and $\hat B$ for the fattened hyperplane given in
Figure \ref{triangleplane}.
We now claim that
\begin{equation*}
	G(D, \hat B) \leq 2\omega_{n-1} \left(R -2R\ln(l/2) \right) \,,
\end{equation*}
which together with (\ref{temp2}) and (\ref{temp3}) will imply the assertion.
To prove the claim, we approximate from inside $D$ by $2H$ strips , $H \in \mathbb{N}$,  of width
$\frac{R}{2H}$ and obtain the energy
\begin{align*}
	&\sum_{h = 1}^H 2G\left( \left( \frac{(h-1)R}{2H} , \frac{hR}{2H}\right) \times
		 \left( \frac{hl}{2H} , \frac{l}{2}\right) , \R \times \left( -\frac{l}{2},0 \right)\right) \\
	&\leq \sum_{h = 1}^H 2G\left( \left( \frac{(h-1)R}{2H} , \frac{hR}{2H}\right) \times
		 \left( \frac{hl}{4H} , \frac{l}{2}\right) , \R \times \left( -\frac{l}{2},-\frac{hl}{4H} \right)\right) \\
	&= \sum_{h = 1}^H 2\omega_{n-1} \frac{R}{H} \left( 2\ln\left( \frac{hl/H + 2l}{8} \right)
		 - \ln\left( \frac{hl}{4H} \right) - \ln(l/2) \right) \\
	&\leq 2\omega_{n-1} \int_{(0,R)} -\ln\left( \frac{l}{4R} \rho \right) - \ln(l/2)\,d\rho \\
	&= 2\omega_{n-1} \left( R -2R\ln(l/2) + R \ln(2) \right) \,,
\end{align*}
where the  first equality follows by Lemma \ref{cylinderplane} and
second inequality follows from $l \leq 8/3$. Letting $H$ tend to $+ \infty$ we obtain the claim.
\end{proof}

\begin{rem} [Estimate of $G$ for a cylinder and the complement of a cone inside a fattened hyperplane]
	\label{cylindercone}
With the same proof we can estimate the energy of a cube and the complement of a pyramid
inside $B_\infty$:
\begin{equation*}
	G\left( A_R, \left(\conv\left( Q_R'(0) \cup \left\{(-l/2,0)\right\} \right)\right)^C \cap B_\infty \right) \leq 
		C(N) R^{N-1} \left( 1 - \ln(l/2) \right) \,,
\end{equation*}
where the only difference is that after the reduction to dimension two we
get the additional triangle $\mathcal R$ indicated in Figure \ref{subdivision}, which turns
the estimate (\ref{temp3}) in
\begin{equation*}
	G(a,B) \leq \frac{3}{2} G(D, \hat B) \,.
\end{equation*}
\end{rem}

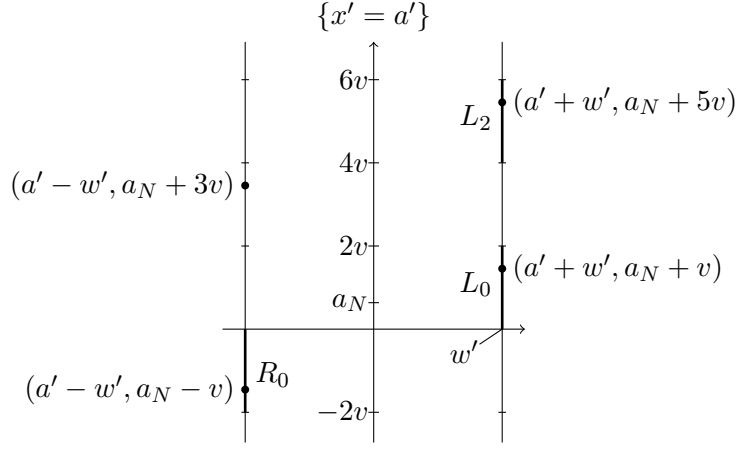
\begin{figure}
\begin{tikzpicture}
	\draw[->](-2,0) --(2,0);
	\draw[->](0,-1.5) --(0,3.8) node[anchor=south]{$\{ x' = a' \}$};
	\draw (1.4,-0.2) --(1.7,0);
	\draw (1.2,-0.3) node[anchor=center]{$w'$};
	\draw (-2pt,3.3) -- (2pt,3.3) node[anchor=east]{$6v$};
	\draw (-2pt,2.2) -- (2pt,2.2) node[anchor=east]{$4v$};
	\draw (-2pt,1.1) -- (2pt,1.1) node[anchor=east]{$2v$};
	\draw (-2pt,-1.1) -- (2pt,-1.1) node[anchor=east]{$-2v$};
	\draw(-1.7,-1.5) --(-1.7,3.8);
	\draw (-1.75 ,3.3) -- (-1.65, 3.3);
	\draw (-1.75 ,2.2) -- (-1.65, 2.2);
	\draw (-1.75 ,1.1) -- (-1.65, 1.1);
	\draw (-1.75 ,-1.1) -- (-1.65, -1.1);
	\draw(1.7,-1.5) --(1.7,3.8);
	\draw (1.75 ,3.3) -- (1.65, 3.3);
	\draw (1.75 ,2.2) -- (1.65, 2.2);
	\draw (1.75 ,1.1) -- (1.65, 1.1);
	\draw (1.75 ,-1.1) -- (1.65, -1.1);
	\draw[line width = 1 pt] (-1.7,-1.1) --(-1.7,0);
	\draw (-1.7,-0.6)node[anchor = west]{$R_0$};
	\draw[line width = 1 pt] (1.7,2.2) --(1.7,3.3);
	\draw (1.7,2.8)node[anchor = east]{$L_2$};
	\draw[line width = 1 pt] (1.7,0) --(1.7,1.1);
	\draw (1.7,0.6)node[anchor = east]{$L_0$};
	\draw (-2pt,0.35) -- (2pt,0.35) node[anchor=east]{$a_N$};
	\fill (-1.7,-0.8) circle [radius = 1.5pt];
	\draw (-1.7,-0.8) node[anchor = east]{$(a' - w', a_N-v)$};
	\fill (1.7,0.8) circle [radius = 1.5pt];
	\draw (1.7,0.8) node[anchor = west]{$(a'+w', a_N+v)$};
	\fill (-1.7,1.9) circle [radius = 1.5pt];
	\draw (-1.7,1.9) node[anchor = east]{$(a'- w', a_N +3v)$};
	\fill (1.7,3) circle [radius = 1.5pt];
	\draw (1.7,3) node[anchor = west]{$(a' + w', a_N+5v)$};
\end{tikzpicture}
\caption{Example of the situation of Lemma \ref{correction_bound} for $i = 0$ and $j = 2$}
\label{reflections}
\end{figure}

\subsection{Lower bounds for the nonlocal energy} \label{temp7}
In this subsection we estimate from below $G(A,B)$ for two disjoint sets 
$A,B$ filling a cylinder up to a set of small measure.
A typical situation in which we would like to estimate $G$ is illustrated in Figure \ref{shading}.
In  Lemma \ref{correction_bound} we consider one-dimensional sections $R_i$ and $L_{i+2j}$
of length $2v$ of a cylinder $S' \times (-l/2,l/2)$; here the points of $R_i \times L_i$
form pairs of points of vertical distance $2v$ and
$L_{i+2h}$ is the vertical translation of $L_i$ by $2h \cdot 2v$;
an illustration is given in Figure \ref{reflections}.
Geometrically speaking, the assertion of Lemma \ref{correction_bound} states that,
if the measure of $(A \cup B)^C$ is small, then the energy of pairs of points which have vertical
distance $2v + 2j \cdot 2v$ can be compared with the energy
of the same points shifted in such a way that vertical distance becomes $2v$.
The estimate of Lemma \ref{correction_bound} will be applied in Lemma \ref{correction_cylinderbound1}.
Therein we compare the energy of $G(A,B)$ on a cylinder
$S' \times (-L/2,L/2),$ where on each vertical section
in an upper portion of the cylinder the average density of $A$ is high,
in a lower portion of the cylinder the average density of $B$ is high
and the $\h^1$- measure of $(A \cup B)^C$ is small, with the energy
$G(A_\infty, B_\infty, S' \times (-L/2,L/2))$ of two fattened hyperplanes inside the cylinder.
Here, the thickness of $A_\infty,B_\infty$ is smaller than the length of the sections with 
high average density of $A, B$, while their distance is comparable to 
the $\h^1$- measure of the sections of $(A \cup B)^C$.
Finally, the estimate from Lemma
\ref{correction_cylinderbound1} is applied in Corollary \ref{correction_cylinderbound2},
the only difference being that now we demand bounds on the volume average density of $A$ in an
upper portion of the cylinder and the volume average density of $B$ on a lower part of the cylinder.

\begin{lem}[Energy lower bound] \label{correction_bound}
Let $S' \subset \R^{N-1}$ be a measurable set, $l > 0$ and $A,B$ be two disjoint subsets of
$S' \times (-l/2,l/2)$ such that for some $d > 0$ it holds
\begin{equation*}
	\h^1 \left( (A \cup B) ^C_{x'} \right) < d \quad \text{for all } x' \in S'\,.
\end{equation*} 
Then, given $u, v > 0$, $a' \in \R^{N-1}$, and for any $w' \in \partial B_u'(0)$ given
$i  \in \mathbb{Z}$ and $j  \in \mathbb{N}$
\footnote{Note that $i$ and $j$ may and in general will depend on $w'$.}
 which satisfy
$i \in \left( \frac{1}{2} - \frac{l}{4v}, \frac{l}{4v} - \frac{1}{2}   \right)$ and
$i + 2j \in \left( \frac{1}{2} - \frac{l}{4v}, \frac{l}{4v} - \frac{1}{2}  \right)$, it holds that
\begin{align*}
	&\int_{\partial B_u'(0)} \int_{(-l/2,l/2)} \chi_A (a' + w', a_n + v) \chi_B(a'-w', a_N -v) \,da_N 
		\,d\h^{N-2}(w') \\
	&\quad \geq \int_{\partial B_u'(0)} \chi_{S'}(a' + w') \chi_{S'}(a' - w') \Bigg(
		\int_{(-v,v)} \chi_{\multimap(R_i \cap B)}(a_N) \chi_{\multimap(L_{i+2j} \cap A)}(a_N) 
		\,da_N - 2 d  \Bigg) \,d\h^{N-2}(w') \\
	&\text{and} \\
	&\int_{\partial B_u'(0)} \int_{(-l/2,l/2)} \chi_B (a' + w', a_n + v) \chi_A(a'-w', a_N -v) \,da_N
		 \,d\h^{N-2}(w') \\
	&\quad \geq \int_{\partial B_u'(0)} \chi_{S'}(a' + w') \chi_{S'}(a' - w') \Bigg(
		\int_{(-v,v)} \chi_{\multimap(R_i \cap A)}(a_N) \chi_{\multimap(L_{i+2j} \cap B)}(a_N)
		\,da_N - 2 d  \Bigg) \,d\h^{N-2}(w') \,,
\end{align*}
where we defined
\begin{equation*}
	R_h := \{ a' -w' \} \times (h2v -2v, h2v)\,, \quad L_h := \{ a' +w' \} \times (h2v, h2v + 2v)
		\quad \text{for } h \in \mathbb{Z}
\end{equation*}
and used the symbol $\multimap (R_h \cap A)$ to denote the image of $R_h \cap A$ under the translation
which maps $R_h$ to an interval centered at zero, and similarly for the other sets.
\footnote{More precisely, we set $\multimap (R_h \cap A)$ for the image of $R_h \cap A$ under the function
$(x',x_N) \mapsto (0, x_N  -h2v + v)$ and $\multimap (L_h \cap B)$ for the image of 
$L_h \cap B$ under the function $(x',x_N) \mapsto (0, x_N  -h2v - v)$. This means in particular that
both $\multimap (R_i \cap A)$ and $\multimap (L_{i+2j} \cap B)$ are subsets of $\{0\} \times (-v,v)$,
which we identify with the interval $(-v,v) \subset \R^1$.}
\end{lem}
\begin{proof}
We only show the first inequality, the second one being symmetric.
We have
\begin{align*}
	&\int_{\partial B_u'(0)} \int_{(-l/2,l/2)} \chi_A (a' + w', a_n + v) \chi_B(a'-w', a_N -v) 
		\,da_N \,d\h^{N-2}(w') \\
	&\quad \geq \int_{(-v,v)} \sum_{h = - \lfloor \frac{l}{4v} - \frac{1}{2} \rfloor}
		^{\lfloor \frac{l}{4v} - \frac{1}{2} \rfloor} \int_{\partial B_u'(0)} 
		\chi_A (a' + w', a_n + v + 2hv) \chi_B(a'-w', a_N - v + 2hv) \,da_N \,d\h^{N-2}(w')  \\
	&\quad = \int_{(-v,v)} \sum_{h = - \lfloor \frac{l}{4v} - \frac{1}{2} \rfloor}
		^{\lfloor \frac{l}{4v} - \frac{1}{2} \rfloor} \int_{\partial B_u'(0)} 
		\chi_A (a' + (-1)^{h+i} w', a_n + v + h2v) \\
	&\quad\quad\quad \chi_B(a'- (-1)^{h+i}w', a_N - v + h2v) \,da_N \,d\h^{N-2}(w')  \\
	&\quad \geq \int_{(-v,v)} \int_{\partial B_u'(0)}
		\sum_{h = i} ^{i + 2j}  
		\chi_A (a' + (-1)^{h+i}w', a_n + v + h2v) \\
	&\quad\quad\quad \chi_B(a'- (-1)^{h+i}w', a_N - v + h2v) \,da_N \,d\h^{N-2}(w')\,.
\end{align*}
Let us now consider the case that  $a' \pm w' \in S'$ and that $(a'- w', a_N - v + i2v) \in B$. If
\begin{equation*}
	\sum_{h = i} ^{i + 2j}  
		\chi_A (a' + (-1)^{h+i}w', a_n + v + h2v) \chi_B(a'- (-1)^{h+i}w', a_N - v + h2v) = 0 \,,
\end{equation*}
then either at least one of the points $(a' + (-1)^{h+i}w', a_n + v + h2v), (a'- (-1)^{h+i}w', a_N - v + h2v)$
lies in $(A \cup B)^C$ (and this happens, on both the one-dimensional sections
$(S \times (-l/2,l/2))_{a' \pm w'}$ only on a set of $\h^1$-measure bounded by $d$),
or all of the points lie in $B$
 \footnote{We use  the fact that the first point for the summand $h$
 equals the second point for the summand $h+1$,
 namely $(a' + (-1)^{h+i}w', a_n + v + h2v) = (a'- (-1)^{h+1+i}w', a_N - v + (h+1)2v)$.
 If the sum equals zero, then all the summands are zero and we already excluded the case that at
 least one of the points lies in $(A \cup B)^C$.}
, and in particular $(a'+ w', a_N + v + (i+2j)2v) \in B$.
This means that this sum is positive if $(a'+ w', a_N + v + (i+2j)2v) \in A$
and none of the points lies in $(A \cup B)^C$ and therefore
\begin{align*}
	&\int_{(-v,v)} \sum_{h = i} ^{i + 2j}  
		\chi_A (a' + (-1)^{h+i}w', a_n + v + h2v) \chi_B(a'- (-1)^{h+i}w', a_N - v + h2v) \, da_N \\
	&\quad \geq \chi_{S'}(a' + w') \chi_{S'}(a' - w') \Bigg(
		\int_{(-v,v)} \chi_{\multimap(R_i \cap B)}(a_N) \chi_{\multimap(L_{i+2j} \cap A)}(a_N) 
		\,da_N - 2 d  \Bigg) \,,
\end{align*}
concluding the proof.
\end{proof}

\begin{lem}[Lower bound of $G$ on a special cylinder I]\label{correction_cylinderbound1}
Given a measurable set $S' \subset \R^{N-1}$, numbers $d,l,L,r > 0$  such that $l < 3r$
and $A,B$ two disjoint subsets of $S' \times (-L/2,L/2)$ such that for some $\Lambda > 0$
\begin{align*}
	&\frac{\h^1 \left( A_{x'} \cap (L/2 -r, L/2) \right)}{r} > \Lambda + \frac{2l}{r} \,, \\
	&\frac{\h^1 \left( B_{x'} \cap (-L/2 , -L/2 + r) \right)}{r} > \Lambda + \frac{2l}{r} \, \\
	& \h^1 \left( (A \cup B)^C_{x'} \right) < \frac{(3\Lambda - 2)d}{2} \quad \text{for all } x' \in S' \,,
\end{align*}
then it holds that
\begin{equation*}
	G(A,B) \geq (3\Lambda -2) G(A_\infty, B_\infty, S' \times \R) \,,
\end{equation*}
where $A_\infty = \R^{N-1} \times (d/2,l/2)$ and $B_\infty = \R^{N-1} \times (-l/2,-d/2)$.
\end{lem}
\begin{proof}
Let us use the notation from the statement of Lemma \ref{correction_bound} and fix any
 $a',w' \in \R^{N-1}$ satisfying $a' \pm w' \in S'$ and $v \in (0, l/2))$. 
Let us write $H_B$ for the set of indices $h \in \mathbb{Z}$ for which 
$R_h$ is contained in $\{a' -w'\}\times (-L/2,-L/2 + r)$;
these cover $\{a' -w'\}\times (-L/2,-L/2 + r)$ up to a set of measure at most $2l$.
Similarly, we write $H_A$ for the set of indices $h \in \mathbb{Z}$ for which 
$L_h$ is contained in $\{a' + w'\}\times (L/2-r,L/2)$;
these cover $\{a' +w'\}\times (L/2-r,L/2)$ up to a set of measure at most $2l$. This means that
\begin{equation*}
	\frac{\h^1\left( B_{a' - w'} \cap \bigcup_{h \in H_B} R_h  \right)}
		{\h^1\left(\bigcup_{h \in H_B} R_h  \right)} > \Lambda \quad \text{and} \quad
	\frac{\h^1\left( A_{a' + w'} \cap \bigcup_{h \in H_A} L_h  \right)}
		{\h^1\left(\bigcup_{h \in H_A} L_h  \right)} > \Lambda \,.
\end{equation*}
Therefore and because $l < 3r$, there exist $i \in H_B$ and $i+2j \in H_A$ such that  for $R_i =: R(a' -w')$
and for $L_{i+2j} =: L(a'+w')$ we have
\begin{equation*}
	\frac{\h^1\left( B_{a' - w'} \cap R_i  \right)}
		{\h^1(R_i)} > \Lambda \quad \text{and} \quad
	\frac{\h^1\left( A_{a' + w'} \cap L_{i+2j}  \right)}
		{\h^1 ( L_{i+2j})} > 2\Lambda -1 \,,
\end{equation*}
which implies
\begin{equation} \label{correction_temp1}
	\frac{\h^1 \left( \multimap (B \cap R(a' -w')) \cap \multimap(A \cap L(a'+w')) \right)}{2v} > 3 \Lambda -2
			\,.
\end{equation}
We estimate as in Lemma \ref{planeplane}
\begin{align*}
	&G(A,B) \geq \int_{\R^{N-1}} \int_{(0,\infty)} 2^N \int_{(d/2, l/2)} \frac{1}{(2 |u,v|)^{N+1}}
		\int_{(-L/2,L/2)} \\
	&\quad\quad \int_{\partial B_u'(0)} \chi_A(a' + w', a_N + v) \chi_B(a' -w', a_N - v) \,d\h^{N-1}(w')
		\,da_N \,dv \,du \,da' \\
	&\quad \geq \int_{\R^{N-1}} \int_{(0,\infty)} 2^N \int_{(d/2, l/2)} \frac{1}{(2 |u,v|)^{N+1}}
		\int_{\partial B_u'(0)} \chi_{S'}(a' +w') \chi_{S'}(a' -w') \\
	&\quad\quad \Bigg( \int_{(-v,v)}
		\chi_{\multimap (B \cap R(a' -w')) \cap \multimap(A \cap L(a'+w'))} (a_N) 
		\,da_N \,d\h^{N-1}(w') - 2 \frac{(3\Lambda - 2)d}{2} \Bigg) \,dv \,du \,da' \\
	&\quad  \geq (3\Lambda - 2) G(A_\infty, B_\infty, S' \times \R) \,,
\end{align*}
where in the second inequality we applied Lemma \ref{correction_bound}
and in the third inequality (\ref{correction_temp1}).
\end{proof}

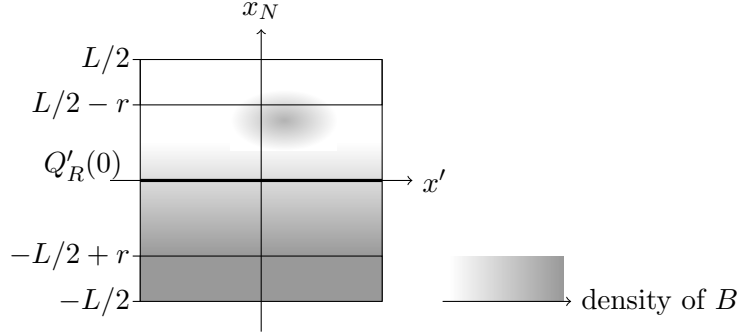
\begin{figure} 
\begin{tikzpicture}
	\shade [shading=axis, top color = white, bottom color = black!40] (-1.6,-1) rectangle (1.6,0.5);
	\shade [shading = radial, outer color = white, inner color = black!30] (-0.4,0.4) rectangle (1,1.2);
	\draw (-1.6,-1.6) -- (-1.6,1.6);
	\draw (1.6,-1.6) -- (1.6,1.6);
	\draw (-1.6,1) -- (1.6,1) -- (1.6,1.6) -- (-1.6,1.6) -- cycle;
	\filldraw [fill = black!40] (-1.6,-1) -- (1.6,-1) -- (1.6,-1.6) -- (-1.6,-1.6) -- cycle;
	\draw[->](-2,0) --(2,0) node[anchor=west]{$x'$};
	\draw[->](0,-2) --(0,2) node[anchor=south]{$x_N$};
	\draw (-1.7,1) -- (-1.6,1) node[anchor=east]{$L/2 -r$};
	\draw (-1.7,-1) -- (-1.6,-1) node[anchor=east]{$-L/2 + r$};
	\draw (-1.7,1.6) -- (-1.6,1.6) node[anchor=east]{$L/2$};
	\draw (-1.7,-1.6) -- (-1.6,-1.6) node[anchor=east]{$-L/2$};
	\draw [line width = 1.2pt] (-1.6,0) -- (1.6,0);
	\draw (-1.7,0.2)  node [anchor = east]{$Q_R'(0)$};
	\shade [shading=axis, left color = white, right color = black!40] (2.5,-1.6) rectangle (4,-1);
	\draw[->](2.4,-1.6) --(4.1,-1.6) node[anchor=west]{density of $B$};
\end{tikzpicture}
\caption{Example of the situation of Corollary \ref{correction_cylinderbound2};
	the density of $A \cup B$ is assumed to be big}
\label{shading}
\end{figure}

An example of a situation in which we would like to apply the following Corollary is given in Figure
\ref{shading}.
\begin{cor}[Lower bound of $G$ on a special cylinder II]\label{correction_cylinderbound2}
Given $r >0$, a cylinder $Q_R'(0) \times (-L/2,L/2)$, a number $\lambda > 0$, a constant $c>0$
and $A,B \subset \R^N$ two disjoint sets such that
\begin{align*}
	&\frac{|Q_R'(0) \times (-L/2,-L/2+r) \cap B|}{|Q_R'(0) \times (-L/2,-L/2+r)|} > 1 - \lambda^2 \,, \\
	&\frac{|Q_R'(0) \times (L/2-r,L/2) \cap A|}{|Q_R'(0) \times (L/2-r,L/2)|} > 1 - \lambda^2 \text{ and} \\
	&|Q_R'(0) \times (-L/2,L/2) \setminus (A \cup B)| < \frac{c(1-3\lambda -6\varepsilon)}
		{2 \lambda_\varepsilon}\,,
\end{align*}
then, for $r \varepsilon < 8/3$, it holds that
\begin{align*}
	&G \left( A,B,Q_R'(0) \times (-L/2,L/2) \right) \\
	&\quad \geq R^{N-1}\omega_{n-1} (1 - 3\lambda  - 6\varepsilon)
		(1 - 2\xi(N-1)(2\lambda + c\varepsilon) - 2\varepsilon)\\
	&\quad\quad\left( \ln\left( \frac{\varepsilon r}{8}\right) - \ln \left( \frac{1}{2 \varepsilon
		\lambda_\varepsilon} \right)
		\right) - R^{N-1} C(N) \left( 1 - \ln\left( \frac{\varepsilon r}{2}\right)  \right) \,,
\end{align*}
where $\xi(N-1)$ and $C(N)$ are dimensional constants.
\end {cor}
\begin{proof}
Let us set
\begin{align*}
	S' := &\Big\{ x' \in Q_R'(0) \,:\, \h^1\left( B_x' \cap (-L/2,-L/2+r) \right) > (1-\lambda)r \text{ and }
		\h^1\left( A_x' \cap (L/2-r,L/2) \right) > (1-\lambda)r \\
		& \text{ and } \h^1\left( (A \cup B)^C_{x'} \right) <
		\frac{1-3\lambda -6\varepsilon}{2 \varepsilon \lambda_\varepsilon} \Big\} \,.
\end{align*}
From our assertions it follows immediately that
\begin{equation*}
	\frac{\h^{N-1}(S')}{\h^{N-1}(Q_R'(0))} > 1 - \lambda -\lambda - c\varepsilon
	= 1 - 2\lambda -c\varepsilon \,.
\end{equation*}
We can now apply Lemma \ref{correction_cylinderbound1} for $l := \varepsilon r$ and
$d = \frac{1}{\varepsilon \lambda_\varepsilon}$ (this means that
$A_\infty = \R^{N-1} \times  \left(  \frac{1}{2\varepsilon \lambda_\varepsilon}, 
\frac{\varepsilon r}{2} \right)$
and
$B_\infty = \R^{N-1} \times  \left( -\frac{\varepsilon r}{2},
 -\frac{1}{2\varepsilon \lambda_\varepsilon} \right)$ )
 to see that
\begin{align*}
	&G \left( A,B,Q_R'(0) \times (-L/2,L/2) \right) \\
	&\quad \geq G \left( A,B,S' \times (-L/2,L/2) \right) \\
	&\quad \geq (3 (1-\lambda -2 \varepsilon) - 2) \, G\left( A_\infty, B_\infty, S' \times \R \right) \\
	&\quad \geq (1-3\lambda -6 \varepsilon)  \left( G \left(A_\infty, B_\infty, Q_R'(0) \times \R \right)
		- G \left(A_\infty \cap {S'}^C \times \R , B_\infty \right) - G \left(A_\infty, B_\infty \cap {S'}^C 
		\times \R\right) \right) \\
	&\quad = (1-3\lambda -6 \varepsilon) \left( G\left(A_\infty, B_\infty, Q_R'(0) \times \R \right)
		- 2G\left(A_\infty \cap {S'}^C \times \R , B_\infty \right)\right) \,.
\end{align*}
We claim that
\begin{equation*}
	G \left(A_\infty \cap {S'}^C \times \R , B_\infty \right) \leq \left( 
		\frac{\xi(N-1) \h^{N-1}({S'}^C) + \varepsilon}
		{\h^{N-1}(Q_R'(0))} \right) G\left(A_\infty, B_\infty, Q_R'(0) \times \R \right) \,.
\end{equation*}
By the Besicovitch covering theorem we can cover $Q_R'(0) \setminus S'$
by $\xi(N-1)$ countable families $(Q_{l,h}')_h$ of pairwise disjoint cubes of side length $r(h,l)$,
$l = 1,\dots,\xi(N-1)$, such that
\begin{equation*}
	\sum_{l = 1}^{\xi(N-1)} \sum_h \h^{N-1}(Q_{l,h}') \leq \xi(N-1) \h^{N-1}(Q_R'(0) \setminus S') 
		+ \varepsilon
\end{equation*}
holds. Now the claim follows by an application of Lemma \ref{cylinderplane} to the pairs
$(A_\infty \cap Q_{l,h}' \times \R, B_\infty)$.
We can use Lemma \ref{cylinderplane} and Lemma \ref{cylindercomplement}
to conclude:
\begin{align*}
	&G \left( A,B,Q_R'(0) \times (-L/2,L/2) \right) \\
	&\quad \geq (1-3\lambda -6 \varepsilon) \bigg( (1-2\xi(N-1)(2\lambda +c\varepsilon) -2\varepsilon) \\
	 &\quad\quad R^{N-1}\omega_{n-1} 
		\left( 2 \ln\left( \frac{\varepsilon r + 1/(\varepsilon \lambda_\varepsilon)}{4} \right) - 
		\ln\left( \frac{\varepsilon r}{2} \right) - \ln\left( \frac{1}{2\varepsilon \lambda_\varepsilon} \right)
 		\right)\bigg) \\
	&\quad \quad - (1-3\lambda -6 \varepsilon) R^{N-1}C(N) \left( 1 -
		 \ln\left( \frac{\varepsilon r}{2} \right) \right) \,.
\end{align*}
Since $2\ln((\varepsilon r)/4) - \ln((\varepsilon r)/2) = \ln((\varepsilon r)/8)$, the assertion follows.
\end{proof}

\subsection{Energy bounds on cubes}
In order to relate the functional $G(A,B)$ with our energy functional $F_\varepsilon(u_\varepsilon)$,
we define a localized version of it:

Given a set $A \subset \R^N$, we write
\begin{equation*}
	F_\varepsilon (u_\varepsilon,A) := \lambda_\varepsilon \int_{\Omega \cap A} W(u_\varepsilon)  \, dx +
		\varepsilon \int_{(\Omega \cap A) \times (\Omega \cap A)}
		\frac{(u_\varepsilon(y) - u_\varepsilon(x))^2} {|y - x|^{N+1}} \, d(y,x)
\end{equation*}
Given a fixed $\delta \in \left( 0, (\beta - \alpha)/2 \right)$, we also define
\begin{align*}
	&A_\varepsilon := \{ x \,:\, u_\varepsilon(x) < \alpha + \delta \} \\
	&B_\varepsilon  := \{ x \,:\, u_\varepsilon(x) > \beta - \delta \} \,.
\end{align*}
We also set for any open set $A$ and $u \in L^1(A)$
\begin{equation*}
	F(u,A) := \left\{
	\begin{array}{lr}
		2(\beta - \alpha)^2\omega_{n-1} k \h^{N-1}\llcorner_A (S_u) \,, &u \in \BV(A, \{ \alpha, \beta \}) \\
		+\infty, &\text{otherwise.}
	\end{array} \right.
\end{equation*}

\begin{lem}[Energy lower bound on a special cube]\label{correction_boundaway0}
Given a family $(u_\varepsilon)$ such that $F_\varepsilon(u_\varepsilon) \leq C < \infty$ 
uniformly as $\varepsilon \rightarrow 0^+$, and given a cube $Q_R(x)$ such that
\begin{equation} \label {correction_temp2}
	|A_\varepsilon \cap Q_R(x)| > a |Q_R(x)| \quad \text{and} \quad |B_\varepsilon \cap Q_R(x)| > a |Q_R(x)| 
\end{equation}
for some $a >0$ and all $\varepsilon > 0$ sufficiently small, then there exists a constant
$C(a)$ (depending also on $N,\alpha,\beta,k$) such that for $\varepsilon$ sufficently small it holds
\begin{equation*}
	F_\varepsilon(u_\varepsilon, Q_R(x)) \geq C(a) R^{N-1}\,.
\end{equation*}
\end{lem}
\begin{proof}
In order to simplify the notation we assume that $x = 0$.
Let us consider a partition of $\R^N$ into $\varepsilon^2$-cubes on an $\varepsilon^2$-lattice and call
$Q_i, i \in H$, the cubes which lie inside $Q_R(0)$. It holds that if
$|A_\varepsilon \cap Q_i| > \varepsilon^{2N + 1/3}$ and $|B_\varepsilon \cap Q_i| > \varepsilon^{2N + 1/3}$,
then
\begin{equation*}
	N^{(N+1)/2} G(A_\varepsilon, B_\varepsilon, Q_i) > \frac{1}{\varepsilon^{2N + 2}}
	\left( \varepsilon^{2N + 1/3} \right)^2 = \varepsilon^{2N - 4/3}
\end{equation*}
and from the energy bound $F_\varepsilon(u_\varepsilon) < C$,
since $\inf \{ W(t) \,:\, \alpha + \delta \leq t \leq \beta - \delta \} > 0$,
it follows that the volume of the union of all
such cubes is bounded by$\frac{C/\varepsilon}{\varepsilon^{-4/3}} = C \varepsilon^{1/3}$.
If instead it holds that
$|A_\varepsilon \cap Q_i| < \varepsilon^{2N + 1/3}$ and 
$|B_\varepsilon \cap Q_i| < \Lambda \varepsilon^{2N}$
for some $\Lambda \in (0, 1-\varepsilon^{1/3})$, then
\begin{equation*}
	|Q_i \setminus (A_\varepsilon \cup B_\varepsilon)| > \varepsilon^{2N} (1 - \Lambda - \varepsilon^{1/3})
\end{equation*}
and because from the energy bound $F_\varepsilon(u_\varepsilon) < C$ it follows that
$|Q_R(0) \setminus (A_\varepsilon \cup B_\varepsilon)| < \frac{C}{\lambda_\varepsilon}$, the volume of 
the union of all such cubes is bounded by $\frac{C / \lambda_\varepsilon}{1 - \Lambda - \varepsilon^{1/3}}$.
Because an analogous result holds with the roles of $A_\varepsilon$ and $B_\varepsilon$ interchanged,
the above estimates imply that the set
\begin{equation*}
	\mathbf{Q} := \left\{ Q_i \,:\, i \in H, \frac{|Q_i \cap A_\varepsilon|}{|Q_i|} > \Lambda 
		\text { or } \frac{|Q_i \cap B_\varepsilon|}{|Q_i|} > \Lambda\right\}
\end{equation*}
satisfies that $|Q_R(0) \setminus \cup \mathbf{Q}| \rightarrow 0$ as $\varepsilon \rightarrow 0^+$.
By (\ref{correction_temp2}) and assuming that $\Lambda$ is big enough, e.g. $\Lambda > 1 - a/2$,
we can then find a direction, which for notational convenience 
we assume to be $e_N$, and a constant $c(a)$ such that there exists an index set $H'$ and a collection
$\{Q_j' \,:\, j \in H'\}$ of
$\varepsilon^2$-cubes on an $\varepsilon^2$-lattice in $\R^{N-1}$ inside $Q_R'(0)$ such that
\begin{align*}
	&\h^{N-1}(\bigcup_{j \in H'}Q_j') > c(a) \h^{N-1}(Q_R'(0)) \quad
		 \text{and any of the cylinders } Q_j' \times (-R/2,R/2) \\
	&\quad \text{ contains both cubes of the types } \frac{|Q_i \cap A_\varepsilon|}{|Q_i|} > \Lambda
		\text{ and } \frac{|Q_i \cap B_\varepsilon|}{|Q_i|} > \Lambda \,.
\end{align*}
We can now apply Corollary \ref{correction_cylinderbound2} for $R = r = \varepsilon^2$,
$1-\lambda^2 = \Lambda$ (and $c = \frac{2C}{1-3\lambda -6\varepsilon}$) to see that for any $j \in H'$
\begin{align*}
	&\frac{\varepsilon}{\varepsilon^{2(N-1)}} G(A_\varepsilon, B_\varepsilon,Q_j' \times (-R/2,R/2)) \\
	&\quad \geq \varepsilon \omega_{n-1} \left(1 - 3 \sqrt{1 - \Lambda}  - 6\varepsilon \right)
		\left(1 - 2\xi(N-1) \left(2 \sqrt{1 - \Lambda} + c\varepsilon \right) - 2\varepsilon \right)\\
	&\quad\quad\left( \ln\left( \frac{\varepsilon^3}{8}\right) - \ln \left( \frac{1}{2 \varepsilon
		\lambda_\varepsilon} \right)
		\right) - \varepsilon C(N) \left( 1 - \ln\left( \frac{\varepsilon^3}{2}\right)  \right) \\
	&\quad \longrightarrow \omega_{n-1} \left(1 - 3 \sqrt{1 - \Lambda} \right)
		\left(1 - 4\xi(N-1) \sqrt{1 - \Lambda} \right) k \quad \text{as } \varepsilon \rightarrow 0^+ \,.
\end{align*}
Here we used that $\varepsilon \ln(\lambda_\varepsilon)$ converges to $k$ as $\varepsilon \rightarrow 0^+$,
and we may assume that the left hand side of the above inequality is at least $\omega_{n-1}/2 k$ for a suitable choice of 
$\Lambda$ and $\varepsilon$ small enough.
Then it follows for sufficently small $\varepsilon$ that
\begin{align*}
	F_\varepsilon(u_\varepsilon, Q_R(0)) &\geq \sum_{j \in H'} \varepsilon G(A_\varepsilon, B_\varepsilon,
		Q_j' \times (-R/2,R/2)) \\
	&\geq \sum_{j \in H'} \varepsilon^{2(N-1)} \frac{\omega_{n-1}k}{2}
		> c(a) \h^{N-1}(Q_R'(0))\frac{\omega_{n-1}k}{2} \,.
\end{align*}
\end{proof}

\begin{lem}[Energy upper bound for a special sequence]\label{recoverycube}
Given a cube $Q_R'(x)$, a number $l > 0$ and the family of functions
\begin{equation*}
	u_\varepsilon: Q_R'(x) \times (-l/2,l/2) \rightarrow \R \,,\, y \mapsto \left\{
	\begin{array}{lr}
		\beta, &y_N \in \left(-\frac{l}{2}, - \frac{\varepsilon}{2 \lambda_\varepsilon} \right)\\
		\beta + (\alpha - \beta) \frac{\lambda_\varepsilon}{\varepsilon} 
			\left( y_N + \frac{\varepsilon}{2 \lambda_\varepsilon} \right), 
		& y_N \in \left( -\frac{\varepsilon}{2 \lambda_\varepsilon},
			\frac{\varepsilon}{2 \lambda_\varepsilon} \right) \\
		\alpha, &y_N \in \left( \frac{\varepsilon}{2 \lambda_\varepsilon}, \frac{l}{2} \right)
	\end{array} \right. \,,
\end{equation*}
it holds that $u_\varepsilon$ converges in $ L^1 \left( Q_R'(x) \times (-l/2,l/2) \right)$ to $u$ as
$\varepsilon \rightarrow 0^+$, where
\begin{equation*}
	 u: Q_R'(x) \times (-l/2,l/2) \rightarrow \R \,,\, y \mapsto \left\{
		\begin{array}{lr}
		\beta, &y_N \in (-l/2,0) \\
		\alpha, &y_N \in (0,l/2)
	\end{array} \right.
\end{equation*}
and
\begin{equation*}
	\limsup_{\varepsilon \rightarrow 0^+} F_\varepsilon(u_\varepsilon, Q_R'(x) \times (-l/2,l/2)) \leq
		2(\beta - \alpha)^2 \omega_{n-1} k R^{N-1} \,.
\end{equation*}
\end{lem}
\begin{proof}
 The first assertion is obvious. Let us prove the second assertion.
 To simplify the notation, we assume that $x = 0$ and set
 $d(\varepsilon) := \frac{\varepsilon}{\lambda_\varepsilon}$.
We start with the preliminary estimates
\footnote{In the first inequality we used the fact that $|u_\varepsilon (y) - u_\varepsilon(x)|
\leq (\alpha -\beta) \frac{\lambda_\varepsilon}{\varepsilon}|y-x|$ and the second one follows by the change
of coordinates $y-x \mapsto y$ and the inclusion $Q_R'(x) \subset Q_{2R}'(0)$ together with the
fact that $|y| \geq \max \{|y_N|, |y'|\}$. }
\begin{align*}
	& \int_{\left( Q_R'(0) \times (-d(\varepsilon),d(\varepsilon))\right)
		\times \left(Q_R'(0) \times (-d(\varepsilon),d(\varepsilon))\right)} 
		\frac{\left(u_\varepsilon(y) - u_\varepsilon(x) \right)^2}{|y-x|^{N+1}} \,d(y,x) \\
	&\quad \leq (\beta -\alpha)^2 \frac{\lambda_\varepsilon^2}{\varepsilon^2} 
		\int_{\left( Q_R'(0) \times (-d(\varepsilon),d(\varepsilon))\right) \times \left(Q_R'(0)
		\times (-d(\varepsilon),d(\varepsilon))\right)}
		\frac{\left(y_N -x_N \right)^2}{|y-x|^{N+1}} \,d(y,x) \\
	&\quad \leq  (\beta -\alpha)^2 \frac{\lambda_\varepsilon^2}{\varepsilon^2}
		|Q_R'(0) \times (-d(\varepsilon),d(\varepsilon))| \int_{Q_{2R}'(0)
		\times (-2d(\varepsilon),2d(\varepsilon))}
		\frac{|y_N|^2}{\max \left\{ |y_N|^{N+1}, |y'|^{N+1} \right\} } \,dy \\
	&\quad \leq 2 (\beta -\alpha)^2 R^{N-1} \frac{\lambda_\varepsilon}{\varepsilon} \left(
		\int_{Q_{2R}'(0) \times (-2d(\varepsilon),2d(\varepsilon))
		\cap \{ y_N > |y'| \}} \frac{1}{|y_N|^{N-1}} +
		\int_{Q_{2R}'(0) \times (-2d(\varepsilon),2d(\varepsilon))
		\cap \{ y_N < |y'| \}} \frac{1}{|y'|^{N-1}} \right) \\
	&\quad \leq C (\beta -\alpha)^2 R^{2(N-1)}
\end{align*}
and
\footnote{Again, we use the change of coordinates $y-x \mapsto y$ and in addition the fact that
$\left(Q_R'(0) \times (-l/2,-d(\varepsilon)) \cup (d(\varepsilon),l/2) \right) -x \subset
Q_{2R}'(0) \times (-l/2,l/2) \setminus B_{d(\varepsilon)/2}(0)$ to obatian the first inequality.}
\begin{align*}
	& \int_{\left( Q_R'(0) \times (-d(\varepsilon)/2,d(\varepsilon)/2)\right) \times \left(Q_R'(0)
		\times (-l/2,-d(\varepsilon)) \cup (d(\varepsilon),l/2) \right)}
		\frac{\left(u_\varepsilon(y) - u_\varepsilon(x) \right)^2}{|y-x|^{N+1}} \,d(y,x) \\
	&\quad \leq (\beta - \alpha)^2 \int_{Q_R'(0) \times (-d(\varepsilon)/2,d(\varepsilon)/2)}
		\int_{Q_{2R}'(0) \times (-l/2,l/2)
		\setminus B_{d(\varepsilon)/2} (0)}
		\frac{1}{|y|^{N+1}} \,dy \,dx \\
	&\quad \leq C (\beta-\alpha)^2 R^{N-1} \,.
\end{align*}
Therefore we have
\footnote{We use in the first equality the two preliminary estimates above 
and the fact that the value of the double
integral on $\left( Q_R'(0) \times (-l/2,-d(\varepsilon)/2)\right) \times \left( Q_R'(0)
\times (d(\varepsilon)/2,l/2)\right)$ equals its value on $\left( Q_R'(0) \times (d(\varepsilon)/2,l/2)\right)
\times \left( Q_R'(0) \times (-l/2,-d(\varepsilon)/2)\right)$. From Lemma \ref{cylinderplane} we obtain
the inequality.}
\begin{align*}
	&\limsup_{\varepsilon \rightarrow 0^+} \, \varepsilon
		\int_{\left(Q_R'(0) \times (-l/2,l/2)\right) \times \left(Q_R'(0) \times (-l/2,l/2)\right)}
		\frac{\left(u_\varepsilon(y) - u_\varepsilon(x) \right)^2}{|y-x|^{N+1}} \,d(y,x) \\
	&\quad = \limsup_{\varepsilon \rightarrow 0^+}\,  2\varepsilon
		\int_{\left( Q_R'(0) \times (-l/2,-d(\varepsilon)/2)\right)
		\times \left( Q_R'(0) \times (d(\varepsilon)/2,l/2)\right) }
		\frac{\left(u_\varepsilon(y) - u_\varepsilon(x) \right)^2}{|y-x|^{N+1}} \,d(y,x) \\
	&\quad =  \limsup_{\varepsilon \rightarrow 0^+}\,  2\varepsilon (\beta-\alpha)^2
		G \left(\left( Q_R'(0) \times (-l/2,-d(\varepsilon)/2)\right),
		\left( Q_R'(0) \times (d(\varepsilon)/2,l/2)\right)\right) \\
	&\quad \leq \limsup_{\varepsilon \rightarrow 0^+}\,  2\varepsilon (\beta-\alpha)^2 R^{N-1} \omega_{n-1}
		\left( 2\ln((l+d(\varepsilon))/4) -\ln(l/2) -\ln(d(\varepsilon)/2) \right) \\
	&\quad = 2(\beta-\alpha)^2 \omega_{n-1} k R^{N-1}\,.
\end{align*}
Since
\begin{equation*}
	\limsup_{\varepsilon \rightarrow 0^+} \lambda_\varepsilon \int_{Q_R'(0) \times (-l/2,l/2)} W(u_\varepsilon) \,dx
		\leq \limsup_{\varepsilon \rightarrow 0^+} \lambda_\varepsilon
		\max\{W(t) \,:\, t \in (\alpha,\beta)\} R^{N-1} d = 0 \,,
\end{equation*}
the assertion follows.
\end{proof}

\subsection{Compactness}
In order to prove compactness, we need the following approximation Lemma.
We define for a given $r>0$ the set of cubes with mid-point on a lattice
\begin{equation*}
	\mathbf{Q}_r^0 := \left\{ Q_r(rh) \,:\, h \in \mathbb Z^N \right\} \quad \text{and} \quad
		\mathbf{Q}_r^i := \left\{ Q_r\left( r(h + e_i/2) \right) \,:\, h \in \mathbb Z^N \right\} \,,
		i = 1,\dots,N \,.
\end{equation*}

\begin{lem} \label{finiteperimeter}
Given a measurable set $A \subset \R^N$ such that for some $a \in (0,1/4)$ it holds that
\begin{equation*}
	\sharp \left\{ Q \in \mathbf Q_r^i \,:\, \frac{|A \cap Q|}{|Q|}
		\in (a, 1-a) \right\} < \frac{C}{r^{N-1} } \,,\, i = 0,\dots,N
\end{equation*}
as $r \rightarrow 0^+$ for some constant $C < +\infty$, then $A$ has finite perimeter.
\end{lem}
\begin{proof}
Let us define
\begin{equation*}
	A_r :=  \bigcup \left\{ Q \in \mathbf Q_r^0 \,:\, \frac{|A \cap Q|}{|Q|} \geq 1-a \right\} .
\end{equation*}
The boundary of $A_r$ consists of faces $(F_h)_h$ which are common to pairs of cubes
$(Q_h, \tilde Q_h)$, where $Q_h \in \mathbf Q_r^0$ and $\tilde Q_h = Q_h \pm e_{i(h)}r $ for a suitable
$e_{i(h)}$ and which satisfy
\begin{equation*}
	\frac{|A \cap Q_h|}{|Q_h|} \geq 1-a \quad \text{and} 
		\quad \frac{|A \cap \tilde Q_h|}{|\tilde Q_h|} < 1-a \,.
\end{equation*}
We have $|A \cap \tilde Q_h|/|\tilde Q_h| \in [a,1-a)$ or $|A \cap \tilde Q_h|/|\tilde Q_h| \in [0,a)$
and in the second case it holds 
\begin{equation*}
	|A \cap Q_h \pm e_{i(h)} r/2|/|Q_h \pm e_{i(h)} r/2| \in [1/2 - a,1/2 + a) \subset [a,1-a) \,.
\end{equation*}
This means that every face $F_h$ can be mapped to a cube $Q \in \mathbf Q_r^i$ satisfying
$\frac{|A \cap Q|}{|Q|} \in [a, 1-a)$ and which has a face that is a shift of $F_h$ in a coordinate direction
with distance at most $r/2$. Therefore, by this mapping at most $2 \cdot 2^N$
faces $F_h$  are mapped to the same cube and
\begin{equation*}
	\sum_h \h^{N-1}(F_h) < 2 \cdot 2^N (N+1) \frac{C}{r^{N-1}}  \h^{N-1}(F_h) \leq C \,.
\end{equation*}
We have shown that the family $(A_r)_r$ has uniformly bounded perimeter. Since
\begin{equation*}
	\left| \left\{ Q \in \mathbf Q_r^i \,:\, \frac{|A \cap Q|}{|Q|}
		\in (a, 1-a) \right\}\right|  \longrightarrow 0 \quad  \text{as } r \rightarrow 0^+
\end{equation*}
holds, $(A_r)_r$ approximates $A$ in measure as $r \rightarrow 0^+$,
and the assertion follows by the lower semicontinuity of the perimeter.
\end{proof}

\begin{prop}[Compactness] \label{compactness}
Given sequences $(\varepsilon_h)_h$ and $(u_{\varepsilon_h})_h$ such that $\varepsilon_h \in \R$ satisfy
$\varepsilon_h \rightarrow 0^+$ and $u_{\varepsilon_h} \in L^1(\Omega)$ satisfy
$F_{\varepsilon_h}(u_{\varepsilon_h}) \leq C$ 
for some constant $C \in (0,\infty)$ as $h \rightarrow \infty$, there exists a function
$u \in \BV(\Omega, \{\alpha,\beta\})$ such that a up to subsequences $(u_{\varepsilon_h})_h$
converges to $u$ in $L^1(\Omega)$ as $h \rightarrow \infty$.
\end{prop}
\begin{proof}
In order to simplify the notation, we write $\varepsilon$ instead of $\varepsilon_h$ and do not
relabel subsequences. Using the energy bound we have
\begin{equation*}
	\int_\Omega W(u_\varepsilon) \,dx \leq \frac{C}{\lambda_\varepsilon}
\end{equation*}
tending to zero as $\varepsilon \rightarrow 0^+$. Because $\Omega$ is bounded and $W$ has at least
linear growth at infinity this imples that the norms $\|u_\varepsilon\|_{L^1(\Omega)}$ are uniformly bounded.
Therefore, up to subsequences, $u_\varepsilon$ converges weakly$\ast$ to a measure with Young
measure $\nu_x$ which is a probability measure for a.e. $x \in \Omega$. Moreover for a.e. $x$,
\begin{equation*}
	\int_{\R} W(t) \,d \nu_x(t) = 0 \,,
\end{equation*}
and it exists a function $\theta: \Omega \rightarrow [0,1]$ such that $\nu_x = \theta(x) \delta_\alpha + 
(1- \theta(x)) \delta_\beta$. We define the $L^1(\Omega)$-function
\begin{equation*}
	u:\Omega \rightarrow \R \,,\, x \mapsto \theta(x) \alpha + (1-\theta(x)) \beta
\end{equation*}
and show that $u \in \BV(\Omega, \{ \alpha, \beta \})$:
Given any constant $c>0$, again using the energy bound, the measure of 
$\Omega \setminus \left( \{ u_\varepsilon \in (\alpha -c, \alpha + c) \cup (\beta -c, \beta + c) \} \right)$
tends to zero as $\varepsilon \rightarrow 0^+$ and $u_\varepsilon$ and the truncated function
$(u_\varepsilon \wedge M) \vee -M$, for a sufficiently large constant $M$, are asymptotically
equivalent in $L^1(\Omega)$. From this we obtain that $u_\varepsilon$ and
\begin{equation*}
	v_\varepsilon := \alpha \chi_{A_\varepsilon} + \beta \chi_{\Omega \setminus A_\varepsilon}
\end{equation*}
are asymptotically equivalent in $L^1(\Omega)$. This means that
\begin{equation} \label{temp4}
	\chi_{A _\varepsilon } \weak* \theta \,.
\end{equation}
Let us note that for any $a \in (0,1/6)$ and for  a.e. point $x \in \{ 3a < \theta < 1-3a \}$, by (\ref{temp4}),
we find a radius $r(x) > 0$ such that for all $r < r(x)$ and $\varepsilon = \varepsilon(r)$ sufficiently small
\begin{equation*}
	\frac{\int_{Q_r(x)} \theta \,dy}{|Q_{r}(x)|} \in (2a,1-2a) \quad\text{and}\quad
		\frac{|Q_{r}(x) \cap A_\varepsilon|}{|Q_{r}(x)|} \in (a,1-a) \,.
\end{equation*}

Therefore, if $\{ 3a < \theta < 1-3a \}$ had positive Lebesgue measure, by Vitali covering theorem,
for any $r > 0$, there would exist a family of pairwise disjoint cubes $\left( Q_{r(h)}(x_h) \right)_h $
with $r(h) \leq r$, $|Q_{r(h)}(x_h) \cap A_\varepsilon|/|Q_{r(h)}(x_h)| \in (a,1-a)$ 
and Lebesgue measure  $|\cup_h Q_{r(h)}(x_h)|$ bounded away from zero
as $r \rightarrow 0^+$. This would contradict the fact that by
Lemma \ref{correction_boundaway0} $F_\varepsilon(u_\varepsilon, Q_{r(h)}(x_h)) \geq Cr(h)^{N-1}$
and the equiboundedness of the energies $F_\varepsilon(u_\varepsilon)$; i.e.,
we have shown that there exists a characteristic function $\chi_A$ such that $\theta = \chi_A$ in 
$L^1(\Omega)$ and it remains to show that $A$ has finite perimeter:
let us fix $a \in (0,1/8)$ and consider
\begin{equation*}
	\mathbf{S}_i :=\left\{ Q \in \mathbf Q_r^i \,:\, \frac{|A \cap Q|}{|Q|} \in (2a, 1-2a) \right\}
		\,,\, i = 0,\dots,N \,.
\end{equation*}
Because of (\ref{temp4}) (and because $\mathbf S_i$ is a finite set), for all $\varepsilon$ sufficently small 
\begin{equation*}
	\mathbf{S}_i  \subset \left\{ Q \in \mathbf Q_r^i \,:\,
		\frac{|A_\varepsilon \cap Q|}{|Q|} \in (a, 1-a) \right\}
		\,,\, i = 0,\dots,N \,.
\end{equation*}
By the energy bound and Lemma \ref{correction_boundaway0}, $\mathbf{S}_i$ have cardinality
bounded by $C / r^{N-1}$ and hence Lemma \ref{finiteperimeter} shows that $A$ is of finite perimeter
and $u \in \BV(\Omega, \{ \alpha,\beta \})$.

Finally, because $\|u_\varepsilon\|_{L^1(\Omega)} \rightarrow \|u\|_{L^1(\Omega)}$,
$u_\varepsilon$ converges up to subsequences also strongly in $L^1( \Omega)$ to $u$.
\end{proof}

\subsection{\(\Gamma\)-convergence}
In this section we state and prove our $\Gamma$-convergence result.
The proofs are similar to the ones given in \cite[Sections 4 and 5]{ab}.

\begin{prop}[$\gammaliminf$ inequality] \label{liminf}
Given sequences $(\varepsilon_h)_h$ and $(u_{\varepsilon_h})_h$ such that $\varepsilon_h \in \R$ satisfy
$\varepsilon_h \rightarrow 0^+$ and $u_{\varepsilon_h}$ converge to  a function $u$ in $L^1(\Omega)$
as $h \rightarrow +\infty$, it holds that
\begin{equation*}
	\liminf_{h \rightarrow \infty} F_{\varepsilon_h}(u_{\varepsilon_h}) \geq F(u) \,.
\end{equation*}
\end{prop}
\begin{proof}
Without loss of generality we may assume that
$\liminf_{h \rightarrow \infty} F_{\varepsilon_h}(u_{\varepsilon_h}) < \infty$ and therefore by
Proposition \ref{compactness} $u \in \BV \left( \Omega, \{ \alpha,  \beta \} \right)$.
In order to simplify the notation, we write $\varepsilon$ instead of $\varepsilon_h$ and do not relabel
subsequences. 
Up to resticting to a subsequence, we may assume
\begin{equation*}
	\liminf_{\varepsilon \rightarrow 0^+} F_\varepsilon(u_\varepsilon)
		= 	\lim_{\varepsilon \rightarrow 0^+} F_\varepsilon(u_\varepsilon) \,.
\end{equation*}

Let us set
\begin{equation*}
	g_\varepsilon(x) := \lambda_\varepsilon W(u_\varepsilon)(x) +
		\varepsilon \int_\Omega \frac{(u_\varepsilon(y) - u_\varepsilon(x))^2}
		{|y - x|^{N+1}} \, dy \,.
\end{equation*}
Then there exists a Radon measure $\mu$ on $\Omega$ such that, again up to subsequences,
\begin{equation*}
	g_\varepsilon \mathcal{L}^N \weak* \mu \,.
\end{equation*}
We claim that
\begin{equation*}
	\frac{d \, \mu}{d \, \h^{N-1} \llcorner S_u}(x) \geq  2(\beta - \alpha)^2 \omega_{n-1} k
		 \quad \text{for } \h^{N-1} \text{-a.e. } x \in S_u \,,
\end{equation*}
which will immediately imply our assertion.
It remains to show the claim. Let us use the ad-hoc notation $Q_R^\nu(x)$ for a cube obtained by rotating
$Q_R(x)$ in such a way that it has a face orthogonal to the vector $\nu \in \R^N$,
$Q_R^{\nu +}(x) := \left\{ y \in Q_R^\nu(x) \,:\, \nu \cdot (y-x) > 0 \right\}$
and $Q_R^{\nu -}(x) := \left\{ y \in Q_R^\nu(x) \,:\, \nu \cdot (y-x) < 0 \right\}$.
For $\h^{N-1}$-almost every $x \in S_u$ it holds that (see e.g. \cite[Theroems 2.22 and 3.59]{afp})
\begin{align*}
	&\text{it exists the measure theoretic outer normal } \nu = \nu_{S_u}(x) \,, \\
	&	\begin{array}{ll}
		\frac{1}{|Q_R^\nu (x)|}\int_{Q_R^\nu(x)} u - \left(\alpha \chi_{Q_R^{\nu +}(x)} + \beta
		 	\chi_{Q_R^{\nu -}(x)} \right) \,dy \rightarrow 0 & \text{as} \quad R \rightarrow 0^+ \,,\\[7pt]
		\frac{\h^{N-1} \llcorner S_u (Q_R^ \nu(x))}{R^{N-1}}  \rightarrow 1 & \text{as} 
			\quad R \rightarrow 0^+ \,,\\[5pt]
		\frac{\mu \left( Q_R^ \nu(x) \right)} {\h^{N-1} \llcorner S_u (Q_R^ \nu(x))}
 			\rightarrow \frac{d \, \mu}{d \, \h^{N-1} \llcorner S_u}(x) & \text{as} 
			\quad R \rightarrow 0^+ \,.
	\end{array}
\end{align*}
Moreover, there exists an at most countable set $\mathbf{B}(x)$
(the set of radii for which $\mu(\partial Q_R^\nu (x)) \not= 0$) such that
\begin{equation*}
	\begin{array}{lr}
		\lim_{\varepsilon \rightarrow 0^+} F_\varepsilon(u_\varepsilon, Q_R^ \nu(x)) = 
			\lim_{\varepsilon \rightarrow 0^+}\int_{Q_R^ \nu(x)} g_\varepsilon \,dy =
			\mu(Q_R^ \nu(x)) & \text{if } R \not\in \mathbf{B}(x) \text{ is sufficiently small} \,.
	\end{array}
\end{equation*}
Therefore, given any constant $\lambda > 0$, for $\h^{N-1}$-a.e. $x \in S_u$
there exists a cube $Q_R^\nu (x) \subset \Omega$ such that
\begin{align} \label{temp5}
	&\frac{d \, \mu}{d \, \h^{N-1} \llcorner S_u}(x) + \lambda \geq \frac{1}{R^{N-1}}
		\lim_{\varepsilon \rightarrow 0^+} F_\varepsilon(u_\varepsilon, Q_R^ \nu(x)) \,, \\
	& |Q_R^{\nu +}(x) \, \Delta \, \left(\{ u = \alpha\}  \cap Q_R^{\nu +}(x)\right) |
		< \frac{\lambda^2}{4} R^N \,, \nonumber\\
	& |Q_R^{\nu -}(x) \, \Delta \, \left(\{ u = \beta \}  \cap Q_R^{\nu -}(x)\right)| < \frac{\lambda^2}{4} R^N
		 \nonumber \,,
\end{align}
and therefore for all $\varepsilon$ sufficiently small,
\begin{align*}
	& |Q_R^{\nu +}(x) \cap A_\varepsilon| > (1 - \lambda^2) \frac{R^N}{2} \,, \\
	& |Q_R^{\nu -}(x) \cap B_\varepsilon| > (1 - \lambda^2) \frac{R^N}{2} \,.
\end{align*}
Since $F_\varepsilon(u_\varepsilon)$ is uniformly bounded in $\varepsilon$, we have
for a positive constant $C$ that $|Q_R^\nu(x) \setminus (A_\varepsilon \cup B_\varepsilon)| 
< C/\lambda_\varepsilon$ and in particular there exists a constant $c > 0$
such that
\begin{equation*}
	|Q_R^\nu(x) \setminus (A_\varepsilon \cup B_\varepsilon)| <
		\frac{c(1-3 \lambda -6 \varepsilon)}{2 \lambda_\varepsilon}.
\end{equation*} 
By the invariance of the energy $F_\varepsilon$ under rotation and translation of the domain, we may
from now on assume that $\nu = e_N$ and $x = 0$.
We can now apply Corollary \ref{correction_cylinderbound2} for $L := R$ and $r= R/2$, which implies
\begin{align} \label{temp6}
	&\varepsilon G \left( A,B,Q_R(0) \right) \nonumber\\
	&\quad \geq \varepsilon R^{N-1}\omega_{n-1} (1 - 3\lambda  - 6\varepsilon)
		(1 - 2\xi(N-1)(2\lambda + c\varepsilon) - 2\varepsilon) \nonumber \\
	&\quad\quad\left( \ln\left( \frac{\varepsilon R}{16}\right) - \ln \left( \frac{1}{2 \varepsilon
		\lambda_\varepsilon} \right) \right)  - \varepsilon R^{N-1} C(N) \left( 1
		- \ln\left( \frac{\varepsilon R}{4}\right)  \right) \nonumber\\
	&\quad \longrightarrow R^{N-1}\omega_{n-1} (1 - 3\lambda)
		(1 - 4\xi(N-1)\lambda)k \quad \text{as } \varepsilon \rightarrow 0^+ \,,
\end{align}
where we used that $\varepsilon \ln(\lambda_\varepsilon) \rightarrow k$ as $\varepsilon \rightarrow 0^+$.
Combining (\ref{temp5}) and (\ref{temp6}) we estimate
\begin{align*}
	\frac{d \, \mu}{d \, \h^{N-1} \llcorner S_u}(0) + \lambda &\geq \frac{1}{R^{N-1}}
		\lim_{\varepsilon \rightarrow 0^+} F_\varepsilon(u_\varepsilon, Q_R(0)) \\
	&\geq \frac{1}{R^{N-1}} \liminf_{\varepsilon \rightarrow 0^+} \, \varepsilon
		\int_{(A_\varepsilon \times B_\varepsilon \cup B_\varepsilon \times A_\varepsilon) 
		\cap Q_R(0) \times Q_R(0)}
		\frac{|u_\varepsilon(y) - u_\varepsilon(x)|^2}{|y-x|^{N+1}} \,d(y,x) \\
	&\geq \frac{1}{R^{N-1}} \liminf_{\varepsilon \rightarrow 0^+} \, \varepsilon \, 
		2 (\beta - \alpha - 2 \delta)^2 G (A_\varepsilon, B_\varepsilon, Q_R(0)) \\
	& \geq 2 (\beta - \alpha - 2 \delta)^2 \omega_{n-1} (1 - 3\lambda) (1 - 4\xi(N-1)\lambda)k \,.
\end{align*}
Because $\lambda$ and $\delta$ are arbitray, this implies the claim.
\end{proof}

Let us use the notation from \cite[Section 5]{ab} and say that a set $P \subset \R^N$ is a \emph{polyhedral
set} if it is an open set whith boundary $\partial P$ a Lipschitz manifold which is contained in the union of
finitely many affine hyperplanes, its \emph{faces} to be the intersecions of $\partial P$ with these
hyperplanes and its  \emph{edges} to be the set of points of $\partial P$ belonging to at least two such
hyperplanes.

\begin{prop}[Recovery sequence] \label{recovery}
Given a sequence $(\varepsilon_h)_h$ such that $\varepsilon_h \in \R$ satisfy
$\varepsilon_h \rightarrow 0^+$ and a function $u \in \BV(\Omega, \{ \alpha,\beta \})$, there exists
a sequence $(u_{\varepsilon_h})_h$ which converges to $u$ in $L^1(\Omega)$
as $h \rightarrow +\infty$, and satisfies
\begin{equation*}
	\limsup_{h \rightarrow \infty} F_{\varepsilon_h}(u_{\varepsilon_h}) \leq F(u) \,.
\end{equation*}
\end{prop}
\begin{proof}
In order to simplify the notation, we write $\varepsilon$ instead of $\varepsilon_h$.
\newline
\textbf{Step 1.}
Let us assume that $u$ is of the form
\begin{equation*}
	u = \left( \beta \chi_P+ \alpha \chi_{P^C} \right) |_\Omega \,, 
		\quad \text{where } P \subset \R^N \text{ is a polyhedral set and } \h^{N-1}(\partial \Omega \cap
		\partial P) = 0 \,.
\end{equation*}
We start by defining recovery sequences for restrictions $u|_A$ of $u$ to certain sets $A$:
If it holds that
\begin{equation}\label{temp8}
	A \text{ is a polyhedral set such that } \h^{N-1}(\partial A \cap \partial P) = 0 \,,
\end{equation}
we define
\begin{equation*}
	u_{\varepsilon, A} : A \cap \Omega \rightarrow \R \,,\, x \mapsto \left\{
		\begin{array}{lr}
			\beta, & A \subset P \\
			\alpha, &\text{otherwise}
		\end{array} \right. \,.
\end{equation*}
Then we have
\begin{equation*}
	\limsup_{\varepsilon \rightarrow 0^+} F_{\varepsilon}(u_\varepsilon,A \cap \Omega) = 0 
		\leq F(u,A) \,.
\end{equation*}
In addition $u_{\varepsilon, A}$ converges to $u|_A$ as $\varepsilon \rightarrow 0^+$.
If instead it holds that
\begin{align} \label{temp9}
	&A \text{ is a polyhedral set such that} A \cap \partial P \text{ is exactly one face } F \subset  \partial P
		\nonumber \\
	&\text{ and the orthogonal projection of } A \text{ to the affine hyperplane containing } 
		F \text{ equals } F \,,
\end{align}
we assume for notational convenience that $F \subset \{ x_N  = 0 \}$ and that the outer normal of
$P$ is $e_N$ on $F$ and set
\begin{equation*}
	u_{\varepsilon,A} : A \cap \Omega \rightarrow \R \,,\, x \mapsto \left\{
	\begin{array}{lr}
		\beta, &x_N \in \left(-\infty, - \frac{\varepsilon}{2 \lambda_\varepsilon} \right)\\
		\beta + (\alpha - \beta) \frac{\lambda_\varepsilon}{\varepsilon} 
			\left(x_N + \frac{\varepsilon}{2 \lambda_\varepsilon} \right), 
		& x_N \in \left( -\frac{\varepsilon}{2 \lambda_\varepsilon},
			\frac{\varepsilon}{2 \lambda_\varepsilon} \right) \\
		\alpha, & x_N \in \left( \frac{\varepsilon}{2 \lambda_\varepsilon}, +\infty \right)
	\end{array} \right. \,.
\end{equation*}
Given any $\lambda > 0$, we can find $R >0$ and finitely many pairwise disjoint
$R$-cubes $Q_i' \subset \R^{N-1}$ which cover $F \cap \Omega$ and satisfy
$\h^{N-1} \left( F \cap \Omega \Delta (\cup_i Q_i') \right) < \lambda$ and define
$u_\varepsilon: \cup_i Q_i' \times \R \rightarrow \R$ in the same way as $u_{\varepsilon,A}$.
From Lemma \ref{cylindercomplement} we have that  for all $i$ it holds
\begin{equation*}
	G \left( Q_i' \times (0,4/3), \cup_{j \not = i} Q_j' \times (-4/3, 0) \right) \leq C(R) \,,
\end{equation*}
and from the boundedness of $\Omega$ it folows
\begin{align*}
	&G \left( Q_i' \times (4/3,\infty), \, 
		\cup_{j \not = i} Q_j' \times (-\infty, 0), \, \Omega \right) \leq C \\
	&G \left( Q_i' \times (0,\infty), \,  \cup_{j \not = i} Q_j' \times (-\infty, -4/3) ,\, \Omega \right) \leq C \,.
\end{align*}
Therefore, writing $l/2$ for the diameter of $\Omega$ and using Lemma \ref{recoverycube}, we get
\begin{align*}
	\limsup_{\varepsilon \rightarrow 0^+} F_\varepsilon(u_\varepsilon, \Omega \cap \cup_i Q_i' \times \R) 
		&\leq \sum_i \limsup_{\varepsilon \rightarrow 0^+}
		F_\varepsilon \left( u_\varepsilon, \Omega \cap Q_i' \times \R \right) \\
	&\leq \sum_i \limsup_{\varepsilon \rightarrow 0^+}
		F_\varepsilon \left( u_\varepsilon, Q_i' \times (-l/2,l/2) \right) \\
	&\leq \sum_i 2(\beta -\alpha)^2 \omega_{n-1} k R^{N-1} \\
	&\leq 2(\beta -\alpha)^2 \omega_{n-1} k \left( \h^{N-1}(F) + \lambda \right) \,,
\end{align*}
and we deduce by the arbitraryness of $\lambda$
and the assumption (\ref{temp9}) on $A$
\begin{equation*}
	\limsup_{\varepsilon \rightarrow 0^+} F_\varepsilon(u_{\varepsilon,A}, A \cap \Omega)
		\leq 2(\beta -\alpha)^2 \omega_{n-1} k \h^{N-1}(F) = F(u,A) \,.
\end{equation*}
Note also that $u_{\varepsilon, A}$ converges to $u|_A$ as $\varepsilon \rightarrow 0^+$.
Let us now assume that there are given two disjoint sets $A_1,A_2$, such that both $A_1$ and $A_2$
satisfy (\ref{temp8}) or (\ref{temp9}) and let us set $u_{\varepsilon,A_1},u_{\varepsilon,A_2}$
for the corresponding functions as introduced above. 
We define
\begin{align*}
	&A_{1,\varepsilon} := \{ x \in A_1 \,:\, d(x,\partial A_1) > \varepsilon/(2 \lambda_\varepsilon) \} \,, \\
	&A_{2,\varepsilon} := \{ x \in A_2 \,:\, d(x,\partial A_2) > \varepsilon/(2 \lambda_\varepsilon) \} \,,
\end{align*}
and set
\begin{equation*}
	A_\varepsilon := A_{1,\varepsilon} \cup A_{2,\varepsilon}
		\quad \text{and} \quad \tilde u_\varepsilon: A_\varepsilon \cap \Omega \rightarrow \R, x \mapsto \left\{
	\begin{array}{lr}
			u_{\varepsilon,A_1}(x), & x \in A_{1,\varepsilon} \\
			u_{\varepsilon,A_2}(x), &x \in A_{2,\varepsilon}
		\end{array} \right. \,.
\end{equation*}
Using Remark \ref{cylindercone} we have that
\begin{align*}
	&G \left(A_{1,\varepsilon} \cap \{\tilde u_\varepsilon = \alpha\} , 
		A_{2,\varepsilon} \cap \{\tilde u_\varepsilon > \alpha\} , \Omega \right) < C \,, \\
	&G \left(  A_{1,\varepsilon} \cap \{\tilde u_\varepsilon = \beta\} ,
		 A_{2,\varepsilon} \cap \{\tilde u_\varepsilon < \beta\}, \Omega \right) < C \,,
\end{align*}
by the very definition of $G$ it is
\begin{align*}
	&G \left(  A_{1,\varepsilon} \cap \{\tilde u_\varepsilon = \alpha\} ,
		  A_{2,\varepsilon} \cap \{\tilde u_\varepsilon = \alpha\}, \Omega \right) = 0 \,, \\
	&G \left( A_{1,\varepsilon} \cap \{\tilde u_\varepsilon = \beta\} , 
		 A_{2,\varepsilon} \cap \{\tilde u_\varepsilon = \beta\}, \Omega \right) = 0 \,,
\end{align*}
and the same holds true with the roles of $A_{1,\varepsilon},A_{2,\varepsilon}$ interchanged.
Therefore we have that
\begin{equation*}
	\limsup_{\varepsilon \rightarrow 0^+} F_\varepsilon(\tilde u_\varepsilon, A_\varepsilon \cap \Omega) 
		\leq \limsup_{\varepsilon \rightarrow 0^+} F_\varepsilon
		(\tilde u_{\varepsilon,A_1}, A_{1,\varepsilon} \cap \Omega)
		+ \limsup_{\varepsilon \rightarrow 0^+} F_\varepsilon
		(\tilde u_{\varepsilon,A_2}, A_{2,\varepsilon} \cap \Omega)
		\leq F(u,A)  \,.
\end{equation*}
Because we can find a covering of $\Omega$ with finitely many pairwise disjoint sets $A_1 \dots, A_n$
as in (\ref{temp8})  or (\ref{temp9}),
we can repeat the procedure now using
$A_1, \dots, A_n$ instead of $A_1,A_2$ and find
functions $\tilde u_\varepsilon: \Omega \cap \left(A_{1,\varepsilon} \cup
\dots \cup A_{n,\varepsilon} \right)\rightarrow \R$ such that
\begin{equation*}
	\limsup_{\varepsilon \rightarrow 0^+} F_\varepsilon(\tilde u_\varepsilon) \leq F(u,\Omega) \,.
\end{equation*}
On the set $\Omega \setminus \left(A_{1,\varepsilon} \cup \dots \cup A_{n,\varepsilon} \right)$ we can find
a function $v_\varepsilon: \Omega \setminus \left(A_{1,\varepsilon} \cup \dots \cup A_{n,\varepsilon} \right)
\rightarrow \R$, which is piecewise affine, its restriction $v_\varepsilon|_{\partial A_{i,\varepsilon}}$
coincides with the restritcion $\tilde u_\varepsilon|_{\partial A_{i,\varepsilon}}$ for all $i$, the set where it is non-constant
has measure bounded by $C \frac{\varepsilon^2}{\lambda_\varepsilon^2}$  and its gradient is bounded by
$\tilde C \frac{\varepsilon}{\lambda_\varepsilon}$  (where $C$ and $\tilde C$ do not depend on
$\varepsilon$)
\footnote{This can be seen as follows: for any pair $(A_i, A_j)$ such that a face $F$ of $A_i$ is contained in a
face of $A_j$, we define for all points $x \in F$ which satisfy that
$B_{\varepsilon/ \lambda_\varepsilon}(x) \subset
A_i \cup A_j$ on the segment connecting the point
$(\partial A_{i,\varepsilon} \cap B_{\varepsilon/ (2\lambda_\varepsilon)}(x))$ and the point
$(\partial A_{j,\varepsilon} \cap B_{\varepsilon/ (2\lambda_\varepsilon)}(x))$
the function $v_\varepsilon$ to be the affine interpolation between 
$u_{\varepsilon}(\partial A_{i,\varepsilon} \cap B_{\varepsilon/ (2 \lambda_\varepsilon)}(x))$
and $u_{\varepsilon}(\partial A_{j,\varepsilon} \cap B_{\varepsilon/ (2 \lambda_\varepsilon)}(x))$.
This gives for any point in $\Omega \setminus \left(A_{1,\varepsilon} \cup \dots \cup A_{n,\varepsilon} \right)$
at most one value for $v_\varepsilon$. The remaining set where $v_\varepsilon$ needs to be defined
is contained in a $C \frac{\varepsilon}{\lambda_\varepsilon}$-neighbourhood
of the $(N-2)$-dimensional edges
of the sets $A_1, \dots A_n$ and therefore has Lebesgue measure bounded by
$C \frac{\varepsilon^2}{\lambda_\varepsilon^2}$. On this set,
we choose any  affine interpolation between the boundary values for $v_\varepsilon$ whose gradient on the
affine regions is bounded by $C \frac{\lambda_\varepsilon}{\varepsilon}$ and $C$ does not depend on
$\varepsilon$. Then $v_\varepsilon$ can be non-constant only on this set and on  an
$C \frac{\varepsilon}{\lambda_\varepsilon}$-neighbourhood of the 
$(N-2)$-dimensional edges of $P$ and therefore
on a set of measure bounded by $C \frac{\varepsilon^2}{\lambda_\varepsilon^2}$.} 
.
We can then use Remark \ref{cylindercone} to show that
\begin{align*}
	&G \left(A_{i,\varepsilon} \cap \{\tilde u_\varepsilon > \alpha\} , 
		\{v_\varepsilon = \alpha\} , \Omega \right) < C \,, \\
	&G \left(  A_{i,\varepsilon} \cap \{\tilde u_\varepsilon < \beta\} ,
		 \{v_\varepsilon = \beta\}, \Omega \right) < C
\end{align*}
and by the very definition of $G$ it holds
\begin{align*}
	&G \left(A_{i,\varepsilon} \cap \{\tilde u_\varepsilon = \alpha\} , 
		\{v_\varepsilon = \alpha\} , \Omega \right) = 0 \,, \\
	&G \left(  A_{i,\varepsilon} \cap \{\tilde u_\varepsilon = \beta\} ,
		 \{v_\varepsilon = \beta\}, \Omega \right) = 0\,.
\end{align*}
We define
\begin{equation*}
	u_\varepsilon: \Omega \rightarrow \R, x \mapsto \left\{
	\begin{array}{lr}
			\tilde u_\varepsilon(x), & x \in \cup_{i = 1, \dots ,n}A_{i,\varepsilon} \\
			v_\varepsilon(x), &x \in \Omega \setminus \cup_{i = 1, \dots ,n}A_{i,\varepsilon}
		\end{array} \right.
\end{equation*}
An estimate as in Lemma \ref{recoverycube} shows
\begin{equation*}
\limsup_{\varepsilon \rightarrow 0^+}  \int_{\left\{ \alpha < v_\varepsilon < \beta \right\} \times \Omega} 
		\frac{\left(u_\varepsilon(y) - u_\varepsilon(x) \right)^2}{|y-x|^{N+1}} \,d(y,x) = 0
\end{equation*}
and we obtain that 
\begin{equation*}
	\limsup_{\varepsilon \rightarrow 0^+} F_\varepsilon(u_\varepsilon) \leq F(u,\Omega) \,.
\end{equation*}
Since $u_\varepsilon$ converges to $u$ as $\varepsilon \rightarrow 0^+$,
this concludes the first step of the proof.
\newline
\textbf{Step 2.}
Let us assume that
\begin{equation*}
	u \in \BV \left( \Omega, \{ \alpha, \beta \} \right) \,.
\end{equation*}
Because sets of finite perimeter can be approximated in perimeter with polyhedral sets,
there exists a sequence $(u_h)_h$ of functions of the type from step 1 which approximate $u$
in $\BV(\Omega)$. In particular we have that $u_h$ approximates $u$ in $L^1(\Omega)$ and
$\limsup_{h \rightarrow \infty} F(u_h) \leq F(u)$. A diagonal argument concludes the proof.
\end{proof}


\begin{thebibliography}{FMS}
\bibitem{ab} Alberti, G. and Bellettini, G., A non-local anisotropic
	model for phase transitions: asymptotic behaviour of rescaled energies,
	Euro. Jnl of Applied Mathematics vol. 9,  1998.
\bibitem{ab2} Alberti, G. and Bellettini G., A non-local anisotropic
	model for phase transitions. {I}. The optimal profile problem, Math. Ann. 310, p. 527--560, 1998.
\bibitem{abs} Alberti, G., Bouchitt\'e, G. and Seppecher, P., Un r\'esultat de perturbations
	singuli\`eres avec la norme $H^{1/2}$, C.R. Acad. Sci. Paris, t. 319, S\'erie I, p. 333 --338, 1994.
\bibitem{afp} Ambrosio, L., Fusco, N. and Pallara, D., Functions of
	Bounded Variation and Free Discontinuity Problems, Clarendon Press, Oxford,
	2000.
\bibitem{dm} Dal Maso, G., An introduction to $\Gamma$-convergence, in: Progress in Nonlinear Differential 
	Equations and their Applications 8, Birkh\"auser Boston Inc., Boston, MA, 1993.
\bibitem{m} Modica L., The Gradient Theory of Phase Transitions and the Minimal Interface Criterion, Arch. 
	Rational Mech. Anal. 98, p. 123--142, 1987.
\bibitem{mm} Modica, L. and Mortola, S., Un esempio di $\Gamma$-convergenza, Boll. Un. Mat. Ital. B (5),
	14, p. 285--299, 1977.
\bibitem{sv1} Savin, O. and Valdinoci, E., $\Gamma$-convergence for nonlocal phase transitions,
	Ana. I. H. Poincar\'e AN 29 (2012) 479 -- 500.
\bibitem{sv2} Savin, O. and Valdinoci, E., Density estimates for a variational model driven 
	by the Gagliardo norm, J. Math. Pures Appl. 101, p. 1--26, 2014.
\bibitem{psv} Palatucci, G., Savin, O. and Valdinoci, E., Local and global minimizers for a variational energy
	involving a fractional norm, Annali di Matematica 192, p. 673--718, 2013.
\end{thebibliography}
\end{document}